\documentclass[11pt]{amsart}

\usepackage{a4wide}
\usepackage{amssymb}
\usepackage{mathrsfs}
\usepackage{enumerate}
\usepackage{titletoc}
\usepackage[colorlinks=true, urlcolor=blue, linkcolor=blue, citecolor=black]{hyperref}
\usepackage[nameinlink]{cleveref}

\theoremstyle{plain}
\newtheorem{theorem}{Theorem}[section]
\newtheorem{lemma}[theorem]{Lemma}
\newtheorem{corollary}[theorem]{Corollary}

\theoremstyle{definition}

\numberwithin{equation}{section}

\newcommand{\ve}{\vert}
\newcommand{\Ve}{\Vert}
\renewcommand{\d}{\mathrm{d}}

\title[]{Harnack inequality for Nonlocal operators on Manifolds with nonnegative curvature}

\author{Jongmyeong Kim}
\address{Department of Mathematical Sciences, Seoul National University, Seoul 08826, Republic of Korea}
\email{saw2132@snu.ac.kr}

\author{Minhyun Kim}
\address{Fakult\"at f\"ur Mathematik, Universit\"at Bielefeld, 33615 Bielefeld, Germany}
\email{minhyun.kim@uni-bielefeld.de}

\author{Ki-Ahm Lee}
\address{Department of Mathematical Sciences, Seoul National University, Seoul 08826, Republic of Korea \& Korea Institute for Advanced Study, Seoul 02455, Republic of Korea}
\email{kiahm@snu.ac.kr}


\begin{document}

\begin{abstract}
We establish the Krylov–Safonov Harnack inequalities and H\"older estimates for fully nonlinear nonlocal operators of non-divergence form on Riemannian manifolds with nonnegative sectional curvatures. To this end, we first define the nonlocal Pucci operators on manifolds that give rise to the concept of non-divergence form operators. We then provide the uniform regularity results for these operators which recover the classical results for second order local operators as limits.
\end{abstract}

\maketitle

\section{Introduction}

This paper is concerned with the Harnack inequalities and H\"older estimates for nonlocal equations on Riemannian manifolds with nonnegative curvature. The Harnack inequalities and H\"older estimates for second order local operators have been studied extensively on Riemannian manifolds. We refer the reader to \cite{Yau75, CY75, SC92} for second order operators of divergence form and \cite{Cab97, Kim04, WZ13, KKL14, KL14} for second order operators of non-divergence form. As nonlocal operators have attracted the attention, some of these results have been extended to nonlocal operators in various contexts. For example, the Harnack inequalities and H\"older estimates were established \cite{CKW19} in the framework of Dirichlet form theory on metric measure spaces with the volume doubling property, which include Riemannian manifolds with nonnegative curvature as a special case. Note that this result is appropriate for linear nonlocal operators of divergence form.

The operators under consideration in this paper are nonlinear and of non-divergence form. To the best of author's knowledge, the Krylov–Safonov Harnack inequalities for nonlocal operators were not available on Riemannian manifolds, while they are well-known in the Euclidean spaces \cite{CS09, GS12}. The aim of this work is to establish the Krylov–Safonov Harnack inequalities and H\"older estimates for fully nonlinear nonlocal operators of non-divergence form on Riemannian manifolds with nonnegative sectional curvatures. Since the underlying space is not flat, we focus on how the curvatures affect the regularity properties of solutions to the equations on manifolds.

\subsection{Nonlocal operators on Riemannian manifolds} \label{sec:operator}

There are several ways of understanding nonlocal operators on the Euclidean spaces—via infinitesimal generators of stochastic processes, semigroup and heat kernels, the Dirichlet-to-Neumann map, or generators of Dirichlet forms—each of which has been applied to obtain nonlocal operators on Riemannian manifolds or more abstract spaces in different contexts. Applebaum and Estrade \cite{AE00} suggested the operators of the form
\begin{equation*}
Lu(x) = \int_{T_x M \setminus \lbrace 0 \rbrace} \left( u(\exp_x \xi) - u(x) \right) \nu_x(\mathrm{d} \xi),
\end{equation*}
as infinitesimal generators of isotropic horizontal L\'evy processes on Riemannian manifold $M$ with some symmetry assumption on it, where $T_x M$ is the tangent space at $x \in M$, $\exp$ is the exponential map, and $\nu_x$ is the L\'evy measure.

On the other hand, Banica, Gonz\'alez, and S\'aez \cite{BGS15} provided the representation of the fractional Laplacian
\begin{equation} \label{eq:hyperbolic}
-(-\Delta_{\mathbb{H}^n})^{\sigma/2} u(x) = \mathrm{p.v.} \int_{\mathbb{H}^n} \left( u(z) - u(x) \right) \mathcal{K} \left( d_{\mathbb{H}^n}(z, x) \right) \mathrm{d}z
\end{equation}
on the hyperbolic spaces $\mathbb{H}^n$ with negative constant curvature, where $\mathrm{p.v.}$ denotes the Cauchy principal value, by using the Fourier transform \cite{CS07}. See \cite{BGS15} for the precise definition of the kernel $\mathcal{K}$ in \eqref{eq:hyperbolic}. For more general compact manifolds and non-compact manifolds with Ricci curvature and injectivity radius bounded below, Alonso-Or\'an, C\'ordoba, and Mart\'inez \cite{AOCM18} provided an integral representation of the fractional Laplace–Beltrami operator with an error term using the well known formula
\begin{equation*}
(-\Delta_g)^{\sigma/2} u(x) = \int_0^\infty \left( e^{-t \Delta_g} u(x) - u(x) \right) \frac{\mathrm{d} t}{t^{1+\sigma/2}}, \quad \sigma \in (0,2),
\end{equation*}
and the heat kernel bounds. However, we do not take the operators in \cite{BGS15} and \cite{AOCM18} as our definition because we are going to consider operators in a more specific form.

In the Dirichlet form theory, it is standard to assume that metric measure space $(M, d, \mu)$ satisfies the volume doubling property. In this setting, the fractional Laplacian-type Dirichlet form
\begin{equation*}
\mathcal{E}(u, v) = \iint_{M \times M \setminus \mathrm{diag}} (u(x) - u(z)) (v(x) - v(z)) J(x, z) \, \mu(\mathrm{d}x) \mu(\mathrm{d}z)
\end{equation*}
with
\begin{equation*}
\frac{\lambda}{\mu(B(x, d(x, z))) d(x, z)^\sigma} \leq J(x, z) \leq \frac{\Lambda}{\mu(B(x, d(x, z))) d(x, z)^\sigma}, \quad 0 < \lambda \leq \Lambda,
\end{equation*}
gives rise to the generator of the fractional Laplacian-type \cite{CKW19}. Motivated by the fact that the Riemannian manifolds with nonnegative curvatures are contained within this framework, we are going to modify this generator in order to define non-divergence form operator.

Let $(M, g)$ be a smooth, complete, connected $n$-dimensional Riemannian manifold with nonnegative sectional curvatures. Let $d_x(z) = d(x, z)$ be the Riemannian distance between two points $x$ and $z$ in $M$, and $\mu_g$ be the Riemannian measure induced by $g$. The operator considered in this paper is modeled on the linear operator of the form
\begin{equation} \label{eq:model_operator}
Lu(x) = (2-\sigma) \mathrm{p.v.} \int_M \frac{u(z) - u(x)}{\mu_g(B(x, d_x(z))) d_x(z)^\sigma} \, \d V(z),
\end{equation}
where $\sigma \in (0,2)$ is a constant. The choice of the factor $(2-\sigma)$ in \eqref{eq:model_operator} is now standard to obtain regularity estimates that are robust in the sense that the constants in the estimates remain uniform as $\sigma$ approaches 2 (see \Cref{sec:main_results}).

To define nonlinear operators, let us consider a class $\mathcal{L}_0$ of linear operators of the form
\begin{equation*}
Lu(x) = \mathrm{p.v.} \int_M (u(z) - u(x)) \nu_x(z) \, \d V(z),
\end{equation*}
with density functions $\nu_x$ satisfying
\begin{equation} \label{eq:symmetry}
\nu_x(z) = \nu_x(\mathcal{T}_x(z)) \quad\text{whenever} ~ d_x(z) < \mathrm{inj}(x),
\end{equation}
where $\mathrm{inj}(x)$ is the injectivity radius of $x$ and $\mathcal{T}_x : B(x, \mathrm{inj}(x)) \to B(x, \mathrm{inj}(x))$ is a map given by $\mathcal{T}_x(z) = \exp_z(-\exp_x(z))$, and
\begin{equation} \label{eq:ellipticity}
\lambda \frac{2-\sigma}{\mu_g(B(x, d_x(z))) d_x(z)^\sigma} \leq \nu_x(z) \leq \Lambda \frac{2-\sigma}{\mu_g(B(x, d_x(z))) d_x(z)^\sigma}.
\end{equation}
Whenever we evaluate $Lu$ at $x$, we split the integral as follows: for $R < \mathrm{inj}(x)$,
\begin{align} \label{eq:Lu}
\begin{split}
Lu(x) 
=&~ (2-\sigma) \mathrm{p.v.} \int_{B_R(x)} (u(z)-u(x)) \nu_x(z) \, \d V(z) \\
&~+ (2-\sigma) \int_{M \setminus B_R(x)} (u(z)-u(x)) \nu_x(z) \, \d V(z).
\end{split}
\end{align}
In contrast to the case of Euclidean spaces, the expression
\begin{align} \label{eq:Lu(x)}
\begin{split}
Lu(x) 
=&~ (2-\sigma) \int_{B_R(x)} \delta(u, x, z) \nu_x(z) \, \d V(z) \\
&~+ (2-\sigma) \int_{M \setminus B_R(x)} (u(z) - u(x)) \nu_x(z) \, \d V(z),
\end{split}
\end{align}
where $\delta(u, x, z) = (u(z) + u(\mathcal{T}_x(z)) - 2u(x))/2$ is the second order incremental quotients, is not available in general because $M$ is not a symmetric manifold. Nevertheless, we will see in \Cref{lem:well-definedness} that for $L \in \mathcal{L}_0$, \eqref{eq:Lu} is well-defined when $u$ is bounded in $M$ and $C^2$ in a neighborhood of $x$. Throughout the paper this observation will be used frequently, especially for the squared distance function $d_x^2(z)$.

The extremal operators and elliptic operators are defined in the standard way as follows. To impose ellipticity on operators, we define the {\it maximal} and {\it minimal operators} by
\begin{equation*}
\mathcal{M}^+_{\mathcal{L}_0} u(x) = \sup_{L \in \mathcal{L}_0} Lu(x) \quad\text{and}\quad \mathcal{M}^-_{\mathcal{L}_0} u(x) = \inf_{L \in \mathcal{L}_0} Lu(x).
\end{equation*}
We say that an operator $\mathcal{I}$ is {\it elliptic with respect to $\mathcal{L}_0$} if
\begin{equation*}
\mathcal{M}^-_{\mathcal{L}_0} (u-v)(x) \leq \mathcal{I}(u, x) - \mathcal{I}(v, x) \leq \mathcal{M}^+_{\mathcal{L}_0} (u-v) (x)
\end{equation*}
for every point $x \in M$ and for all bounded functions $u$ and $v$ which are $C^2$ near $x$.

We point out that the usual explicit expressions of extremal operators in the Euclidean spaces
\begin{align*}
\begin{split}
\mathcal{M}^+_{\mathcal{L}_0} u(x) &= (2-\sigma) \int_{\mathbb{R}^n} \frac{\Lambda \delta_+(u, x, y) - \lambda \delta_-(u, x, y)}{\omega_n \ve y \ve^{n+\sigma}} \, \mathrm{d}y \quad\text{and}\\
\mathcal{M}^-_{\mathcal{L}_0} u(x) &= (2-\sigma) \int_{\mathbb{R}^n} \frac{\lambda \delta_+(u, x, y) - \Lambda \delta_-(u, x, y)}{\omega_n \ve y \ve^{n+\sigma}} \, \mathrm{d}y,
\end{split}
\end{align*}
where $\omega_n$ is the volume of the $n$-dimensional unit ball and $\delta(u, x, y) = (u(x+y) + u(x-y) - 2u(x))/2$, are not available on manifolds in general. Thus, whenever we evaluate $Lu$ or $\mathcal{M}^\pm_{\mathcal{L}_0} u$ at $x$, we have to split the integral as \eqref{eq:Lu} or \eqref{eq:Lu3} to compute each integral.

\subsection{Main results} \label{sec:main_results}

The main results are the Krylov–Safonov Harnack inequality and interior H\"older estimates for fully nonlinear nonlocal operators of non-divergence form on Riemannian manifolds with nonnegative sectional curvatures. Throughout the paper we assume that $(M, g)$ is a smooth, complete, connected Riemannian manifold with nonnegative sectional curvatures, satisfying the reverse volume doubling property \eqref{eq:RVD} with constant $a_1$ and the volume comparability \eqref{eq:comparability} with constant $a_2$. See \Cref{sec:preliminaries} for the assumptions \eqref{eq:RVD} and \eqref{eq:comparability}. Let us begin with the Krylov–Safonov Harnack inequality.

\begin{theorem} [Harnack inequality] \label{thm:Harnack}
Let $\sigma_0 \in (0,2)$ and assume $\sigma \in [\sigma_0, 2)$. For $z_0 \in M$, let $K = K_{\mathrm{max}}(B(z_0, \mathrm{inj} (z_0)))$ be the supremum of the sectional curvatures in $B(z_0, \mathrm{inj}(z_0))$ and let $R > 0$ be such that $2R < \mathrm{inj}(z_0) \land \frac{\pi}{\sqrt{K}}$. If $u \in C^2(B_{2R}(z_0)) \cap L^\infty(M)$ is a nonnegative function on $M$ satisfying
\begin{equation*}
\mathcal{M}^-_{\mathcal{L}_0} u \leq C_0 \quad \text{and} \quad \mathcal{M}^+_{\mathcal{L}_0} u \geq -C_0 \quad\text{in}~ B_{2R}(z_0),
\end{equation*}
then
\begin{equation*}
\sup_{B_R(z_0)} u \leq C \left( \inf_{B_R(z_0)} u + C_0 R^\sigma \right)
\end{equation*}
for some universal constant $C > 0$, depending only on $n$, $\lambda$, $\Lambda$, $a_1$, $a_2$, and $\sigma_0$.
\end{theorem}

The next result is the interior H\"older estimate for fully nonlinear nonlocal operators of non-divergence form. In contrast to the case of local operators, it does not immediately follow from the Harnack inequality. In the sequel, $\Ve \cdot \Ve'$ denotes the non-dimensional norm.

\begin{theorem} [H\"older estimates] \label{thm:Holder}
Let $\sigma_0 \in (0,2)$ and assume $\sigma \in [\sigma_0, 2)$. For $z_0 \in M$,let $K = K_{\mathrm{max}}(B(z_0, \mathrm{inj} (z_0)))$ and let $R > 0$ be such that $2R < \mathrm{inj}(z_0) \land \frac{\pi}{\sqrt{K}}$. If $u \in C^2(B_{2R}(z_0)) \cap L^\infty(M)$ is a function on $M$ satisfying
\begin{equation*}
\mathcal{M}^-_{\mathcal{L}_0} u \leq C_0 \quad \text{and} \quad \mathcal{M}^+_{\mathcal{L}_0} u \geq -C_0 \quad\text{in}~ B_{2R}(z_0),
\end{equation*}
then $u \in C^\alpha (\overline{B_R (z_0)})$ and
\begin{equation*}
\Ve u \Ve'_{C^\alpha(\overline{B_R(z_0)})} \leq C \left( \Ve u \Ve_{L^\infty(M)} + C_0 R^\sigma \right)
\end{equation*}
for some universal constants $\alpha \in (0,1)$ and $C > 0$, depending only on $n$, $\lambda$, $\Lambda$, $a_1$, $a_2$, and $\sigma_0$.
\end{theorem}

It is noticeable that the universal constants in \Cref{thm:Harnack} and \Cref{thm:Holder} do not depend on nearby curvature upper bound although they
depend on the lower bound 0. This means that, in particular, when $M = \mathbb{R}^n$, \Cref{thm:Harnack} and \Cref{thm:Holder} provide the results on the Krylov–Safonov Harnack inequality and H\"older estimates as in \cite{CS09} without any restriction on $R$. Moreover generally, the restriction on $R$ disappears when $M$ is a manifold with $\mathrm{inj}(M) = \infty$. In this case, \Cref{thm:Harnack} extends the global Harnack inequality for local operators \cite{Cab97} to nonlocal operators.

Another important feature of \Cref{thm:Harnack} and \Cref{thm:Holder} is the robustness of the estimates. Since the universal constants in the results depend only on $\sigma_0$, not on $\sigma \in (0,2)$ itself, we could get the local Harnack inequality and H\"older estimates for the second order local operators as limit $\sigma \to 2$.

\begin{corollary}
For $z_0 \in M$, let $K = K_{\mathrm{max}}(B(z_0, \mathrm{inj} (z_0)))$ be the supremum of the sectional curvatures in $B(z_0, \mathrm{inj}(z_0))$ and let $R > 0$ be such that $2R < \mathrm{inj}(z_0) \land \frac{\pi}{\sqrt{K}}$.
\begin{enumerate}[(i)]
\item
If $u \in C^2(B_{2R}(z_0)) \cap L^\infty(M)$ is a nonnegative function on $M$ satisfying $\mathcal{M}^-(D^2 u) \leq C_0$ and $\mathcal{M}^+(D^2 u) \geq -C_0$ in $B_{2R}(z_0)$, then
\begin{equation*}
\sup_{B_R(z_0)} u \leq C \left( \inf_{B_R(z_0)} u + C_0 R^2 \right).
\end{equation*}

\item
If $u \in C^2(B_{2R}(z_0)) \cap L^\infty(M)$ is a function on $M$ satisfying $\mathcal{M}^-(D^2 u) \leq C_0$ and $\mathcal{M}^+(D^2 u) \geq -C_0$ in $B_{2R}$, then $u \in C^\alpha (\overline{B_R (z_0)})$ and
\begin{equation*}
\Ve u \Ve'_{C^\alpha(\overline{B_R(z_0)})} \leq C \left( \Ve u \Ve_{L^\infty(M)} + C_0 R^2 \right).
\end{equation*}
\end{enumerate}
The universal constants $C > 0$ and $\alpha \in (0,1)$ depend only on $n$, $\lambda$, $\Lambda$, $a_1$, and $a_2$.
\end{corollary}

Let us make some remarks on the results. It would be the best if we get ABP-type estimate with $L^n$-norm as Cabr\'e proved in \cite{Cab97}. However, we will establish the ABP estimate with 
Riemann sums of $L^\infty$-norm as Caffarelli and Silvestre showed in the Euclidean space \cite{CS09}. To the best of author's knowledge, the full ABP estimate with $L^n$-norm for fully nonlinear operators are not available even in the Euclidean spaces. For the class of operators with additional assumptions, Guillen and Schwab \cite{GS12} provided the ABP estimates using both $L^n$ and $L^\infty$ norms in Euclidean spaces. For this type of estimate, we believe it would be applicable to our case.

For curvature bound and imposed radius condition, we refer to Cabr\'e's observation in the last paragraph of \cite{Cab97}. So, we used the injectivity radius and imposed the condition $15R < \mathrm{inj} (z_0) \wedge \frac{\pi}{\sqrt{K}}$,
for $K=K_{\mathrm{max}}(B(z_0,\mathrm{inj}(z_0)))$, on the radius of the ball. However, it might be more convenient to consider a strongly convex region (or a strongly convexity radius) instead of the injectivity radius: we
call $U \subseteq M$ is strongly convex if every ball $B_\rho(x) \subseteq U$ is
convex. This is because our operators are nonlocal and we need to consider the relation between nearby points. Nevertheless, we will use the injectivity radius because it is more general.

If we assume the global upper bound of the sectional curvature such as $\mathrm{Sect}(g) < K$ on $M$, the radius condition would be reduced to $5R < \frac{\pi}{\sqrt{K}}$. Moreover, the additional assumptions—the reverse volume doubling property \eqref{eq:RVD} and comparability of volumes \eqref{eq:comparability}—on manifold are naturally satisfied.

Since manifold is not symmetric space in general, 
a nonlocal antisymmetric part in the operator appears naturally. 
Because it has no second order incremental quotient of function,
we cannot expect the integrability of operators as usual.
However, due to the smoothness of volume form which exists inherently,
we figure out the antisymmetric part
has the same order as the traditional symmetric part.

Moreover, we want to emphasize that we do not use affine functions and cone
technique as usual because affine functions with arbitrary directions do not
exist on manifold in general. Thus, we use the squared distance function to solve this difficulty. Typically,
when we control the gradient of the envelope $\Gamma$ (defined in \Cref{sec:ABP}) with the squared distance function, we might consider the coarea formula as in \cite{Cab97}. However, since the order of differentiability of nonlocal operators is strictly less than 2, we cannot simply use the coarea formula. At this part, we will directly estimate gradient with Jacobi fields.

Lastly, for further researches, we are expecting that we can get a global Harnack
inequality for restricted manifolds. In general, in this paper, we could not
stretch the radius of ball due to the injectivity issue. We are also expecting
that we could get similar regularity properties for nonlocal operator with
kernels of variable orders, which are studied in \cite{KKL16, KL20, BM20} on
Euclidean spaces.

\subsection{Outline}

This paper is organized as follows. 
In \Cref{sec:preliminaries}, we ensure integrability of operators. We also bound
second difference of squared distance function. Mainly we need this bound for gradient
estimate of the solution. Furthermore, we introduce some definitions and collect some result on dyadic cubes for the analysis on manifold. In \Cref{sec:ABP}, we introduce an envelope defined by squared distance
function and estimate its gradient so that we get a (weak type)
Aleksandrov–Bakelman–Pucci estimate. \Cref{sec:barrier} is devoted to the construction of a barrier function. In \Cref{sec:L_eps_estimate}, $L^\varepsilon$-estimate is established by using the ABP estimate and the barrier function obtained in the previous sections. The proofs for the Harnack inequality and H\"older estimate are provided in \Cref{sec:Harnack} and \Cref{sec:Holder}, respectively.

\section{Preliminaries} \label{sec:preliminaries}

This section is devoted to the basic knowledge on Riemannian geometry that will be useful in the rest of the paper. For more details, the reader may consult \cite{Jos17, CLN06, Vil09}.

Let $(M, g)$ be a smooth, complete manifold of dimension $n$. Let us denote by $R(\xi, \eta) \zeta$ the curvature tensor, then the sectional curvature of the plane determined by linearly independent tangent vectors $\xi, \eta \in T_x M$ is given by
\begin{equation*}
\mathrm{Sect}(\xi, \eta) = \frac{g(R(\xi, \eta) \xi, \eta)}{\ve \xi \ve_g^2 \ve \eta \ve_g^2 - g(\xi, \eta)^2}.
\end{equation*}
Let $d_y(\cdot) := d(\cdot, y)$ be the distance function. We will see that the distance squared function $\frac{1}{2} d_y^2$ will play an important role in the regularity results. Let us collect and study some useful properties of this function. First of all, it is continuous in $M$ and smooth in $M \setminus \mathrm{Cut}_y$. For any $x \notin \mathrm{Cut}_y$, the Gauss lemma implies that
\begin{equation*}
\nabla ( d_y^2 /2 )(x) = - \exp_x^{-1} y.
\end{equation*}
Moreover, it is well-known that if $K_1 \leq \mathrm{Sect} \leq K_2$ in $B_{\mathrm{inj}(y)}(y)$ with $K_1 \leq 0$ and $K_2 \geq 0$, then the Hessian of $d_y^2/2$ has upper and lower bounds
\begin{equation} \label{eq:Hessian_general}
\begin{split}
\sqrt{K_2} d_y(x) \cot \left( \sqrt{K_2} d_y(x) \right) \ve \xi \ve_g^2
&\leq D^2 ( d_y^2 /2 )(x) (\xi, \xi) \\
&\leq \sqrt{-K_1} d_y(x) \coth \left( \sqrt{-K_1} d_y(x) \right) \ve \xi \ve_g^2,
\end{split}
\end{equation}
for $x \in B_\rho(y)$ and $\xi \in T_x M$, where $\rho < \frac{\pi}{2\sqrt{K_2}}$ in case $K_2 > 0$ and $\rho < \mathrm{inj}(y)$ otherwise (see, for example, \cite[Theorem 6.6.1]{Jos17}). Since we are assuming that $\mathrm{Sect} \geq 0$, the bounds \eqref{eq:Hessian_general} read as
\begin{equation} \label{eq:Hessian}
0 \leq \sqrt{K} d_y(x) \cot \left( \sqrt{K} d_y(x) \right) \ve \xi \ve_g^2 \leq D^2 ( d_y^2 /2 )(x) (\xi, \xi) \leq \ve \xi \ve_g^2,
\end{equation}
where $K = K_{\mathrm{max}}(B_{\mathrm{inj}(y)}(y))$ is the supremum of the sectional curvatures in $B_{\mathrm{inj}(y)}(y)$. Using \eqref{eq:Hessian} and the mean value theorem for integrals, we obtain the following lemma.

\begin{lemma} \label{lem:dist_squared}
For any $y \in M$ and $x \in B_\rho(y)$, let $\xi \in T_x M$ be such that $\exp_x(s \xi) \in B_\rho(y)$ for all $s \in (-1,1)$, where
\begin{equation} \label{eq:rho_K}
\rho < 
\begin{cases}
\frac{\pi}{2\sqrt{K}} &\text{if}~ K := K_{\mathrm{max}}(B_{\mathrm{inj}(y)}(y)) > 0, \\
\mathrm{inj}(y) &\text{if}~ K = 0.
\end{cases}
\end{equation}
Then,
\begin{equation*}
0 \leq (1-t) d_y^2(\exp_x(t \xi)) + t d_y^2(\exp_x((1-t)(-\xi))) - d_y^2(x) \leq t(1-t) \ve \xi \ve_g^2
\end{equation*}
for any $t \in (0,1)$.
\end{lemma}

Let us now recall Gromov's theorem in a manifold with a nonnegative Ricci curvature. Since we assume that $\mathrm{Sect} \geq 0$, the Ricci curvature is also nonnegative. The Gromov's theorem says that
\begin{equation*}
\frac{\mu_g(B(x, R))}{\ve B_R \ve} \quad\text{is nonincreasing in}~ R
\end{equation*}
for any $x \in M$, where $\ve B_R \ve$ is the volume in $\mathbb{R}^n$ of a ball of radius $R$. It is known that the ratio approaches 1 as $R$ goes to zero, so together with the monotonicity it implies that $\mu_g(B(x, R)) \leq \ve B_R \ve$. Moreover, the Gromov's theorem also gives rise to the volume doubling property
\begin{equation} \label{eq:VD} \tag{VD}
\frac{\mu_g(B(x, R))}{\mu_g(B(x, r))} \leq \left( \frac{R}{r} \right)^n, \quad 0 < r \leq R.
\end{equation}
The volume doubling property provides the following integrability of kernels $\nu_x$.

\begin{lemma} \label{lem:integrability}
Let $\sigma_0 \in (0,2)$ and assume $\sigma \in [\sigma_0, 2)$. Then,
\begin{equation} \label{eq:integrability}
(2-\sigma) \int_M \left( R^2 \land d_x(z)^2 \right) \frac{\d V(z)}{\mu_g(B(x, d_x(z))) d_x(z)^\sigma} \leq C R^{2-\sigma}
\end{equation}
for some constant $C = C(n, \sigma_0) > 0$.
\end{lemma}

\begin{proof}
By the volume doubling property \eqref{eq:VD}, we have
\begin{align} \label{eq:B_R}
\begin{split}
\int_{B_R(x)} \frac{(2-\sigma) d_x(z)^{2-\sigma}}{\mu_g(B(x, d_x(z)))} \, \d V(z)
&= \sum_{k=0}^\infty \int_{B(x, 2^{-k} R) \setminus B(x, 2^{-(k+1)}R)} \frac{(2-\sigma) d_x(z)^{2-\sigma}}{\mu_g(B(x, d_x(z)))} \, \d V(z) \\
&\leq \sum_{k=0}^\infty \frac{\mu_g(B(x, 2^{-k}R))}{\mu_g(B(x, 2^{-(k+1)}R))} (2-\sigma) 2^{-k(2-\sigma)} R^{2-\sigma} \\
&\leq 2^n \frac{2-\sigma}{1-2^{-(2-\sigma)}} R^{2-\sigma} \leq C(n) R^{2-\sigma},
\end{split}
\end{align}
where we observed in the last inequality that the function $t/(1-2^{-t})$ is bounded in $[0, 2]$ from above. Similarly, by the volume doubling property \eqref{eq:VD} again, we obtain
\begin{align} \label{eq:B_R_complement}
\begin{split}
\int_{M \setminus B_R(x)} \frac{(2-\sigma) R^2}{\mu(B(x, d_x(z))) d_x(z)^\sigma} \, \d V(z)
&= \sum_{k=0}^\infty \int_{B(x, 2^{k+1}R) \setminus B(x, 2^k R)} \frac{(2-\sigma) R^2}{\mu(B(x, d_x(z))) d_x(z)^\sigma} \, \d V(z) \\
&\leq \sum_{k=0}^\infty \frac{\mu_g(B(x, 2^{k+1}R))}{\mu_g(B(x, 2^k R))} (2-\sigma) 2^{-k\sigma} R^{2-\sigma} \\
&\leq 2^n \frac{2 - \sigma}{1-2^{-\sigma}} R^{2-\sigma} \leq \frac{2^{n+1}}{1-2^{-\sigma_0}} R^{2-\sigma}.
\end{split}
\end{align}
Therefore, \eqref{eq:integrability} follows by combining the inequalities \eqref{eq:B_R} and \eqref{eq:B_R_complement}.
\end{proof}

Using \Cref{lem:integrability}, we show that $Lu$ is well-defined.

\begin{lemma} \label{lem:well-definedness}
Let $\sigma_0 \in (0,2)$ and assume $\sigma \in [\sigma_0, 2)$. For $x \in M$, let $K$ be the supremum of the sectional curvatures in $B_{\mathrm{inj}(x)}(x)$ and let $2R < \mathrm{inj}(x) \land \frac{\pi}{\sqrt{K}}$. Then, for $L \in \mathcal{L}_0$ and for $u \in C^2(\overline{B_R(x)}) \cap L^\infty(M)$,
\begin{equation} \label{eq:antisymmetric_part}
\ve Lu(x) \ve \leq C\Lambda \left( \Ve u \Ve'_{C^2(\overline{B_R(x)})} + \Ve u \Ve_{L^\infty(M)} \right) R^{-\sigma},
\end{equation}
where $C = C(n, \sigma_0) > 0$ is a universal constant. Therefore, the value of $Lu$ at $x$ is well-defined.
\end{lemma}

\begin{proof}
By assuming $R$ sufficiently small, we may assume that $u$ is $C^2(\overline{B_R(x)})$ and bounded in $M$.

Let us decompose the measure into the symmetric and antisymmetric parts with respect to $x$, that is, $\d V(z) = \d V_s(z) + \d V_a(z)$, where $\d V_s(z) := \frac{1}{2} (\d V(z) + \d V(\mathcal{T}_x(z)))$ and $\d V_a(z) := \frac{1}{2} (\d V(z) - \d V(\mathcal{T}_x(z)))$. Then, for $L \in \mathcal{L}_0$ we have
\begin{align} \label{eq:Lu3}
\begin{split}
Lu(x)
=~& (2-\sigma) \int_{B_R(x)} \delta(u,x,z) \nu_x(z) \, \d V_s(z) \\
&+ (2-\sigma) \int_{B_R(x)} (u(z) - u(x)) \nu_x(z) \, \d V_a(z) \\
&+ (2-\sigma) \int_{M \setminus B_R(x)} (u(z) - u(x)) \nu_x(z) \, \d V(z) =: I_1 + I_2 + I_3.
\end{split}
\end{align}
We may apply \Cref{lem:integrability} for $\d V_s$ and $\d V$ to obtain $\ve I_1 \ve \leq C\Lambda \Ve u \Ve'_{C^2(\overline{B_R(x)})} R^{-\sigma}$ and $\ve I_3 \ve \leq C\Lambda \Ve u \Ve_{L^\infty(M)} R^{-\sigma}$, respectively. For $I_2$, we observe that
\begin{align*}
\ve I_2 \ve
&\leq \Lambda (2-\sigma) \Ve u \Ve_{L^\infty(B_R)} \int_{B_R(x)} \frac{1}{\mu_g(B(x, d_x(z))) d_x(z)^\sigma} \, \ve \d V_a \ve(z) \\
&\leq \Lambda (2-\sigma) \Ve u \Ve_{L^\infty(B_R)} \int_0^R \int_{\partial B_1} \frac{t^2}{\mu_{g^\ast}(B(0, t)) t^\sigma} \bigg( t^{n-1} - \bigg( \frac{\sin(\sqrt{K}t)}{\sqrt{K}} \bigg)^{n-1} \bigg) \, \d v \, \d t,
\end{align*}
where $g^\ast$ is the induced metric. Note that the inequalities
\begin{equation*}
t^{n-1} - \bigg( \frac{\sin(\sqrt{K}t)}{\sqrt{K}} \bigg)^{n-1} \leq \frac{n-1}{3!} (\sqrt{K}t)^2 t^{n-1}, \quad t \sqrt{K} \leq \sqrt{6},
\end{equation*}
and
\begin{equation*}
t \leq \frac{\pi}{2} \frac{\sin (\sqrt{K}t)}{\sqrt{K}}, \quad t \sqrt{K} \leq \frac{\pi}{2},
\end{equation*}
can be applied to obtain that
\begin{equation*}
t^{n-1} - \bigg( \frac{\sin(\sqrt{K}t)}{\sqrt{K}} \bigg)^{n-1} \leq C Kt^2 \bigg( \frac{\sin(\sqrt{K}t)}{\sqrt{K}} \bigg)^{n-1},
\end{equation*}
since we have assumed that $2R < \pi/\sqrt{K}$. Therefore, 
\begin{equation*}
\ve I_2 \ve \leq C \Lambda (2-\sigma) \Ve u \Ve_{L^\infty(B_R)} \int_{B_R(x)} \frac{K d_x^2(z)}{\mu_g(B(x, d_x(z))) d_x(z)^\sigma} \, \d V(z).
\end{equation*}
By \Cref{lem:integrability}, we arrive at $\ve I_2 \ve \leq C\Lambda \Ve u \Ve_{L^\infty(B_R)} K R^{2-\sigma}$. Again, by using $2R < \pi/\sqrt{K}$,
\begin{equation} \label{eq:I_2_anti}
\ve I_2 \ve \leq C\Lambda \Ve u \Ve_\infty R^{-\sigma}.
\end{equation}
The estimate \eqref{eq:I_2_anti} together with estimates for $I_1$ and $I_3$ finishes the proof.
\end{proof}

Here are two assumptions on manifold we are going to use throughout the paper.
\begin{itemize}
\item
(Reversed volume doubling property) Let us assume that there is a constant $a_1 \in (0, 1]$ such that
\begin{equation} \label{eq:RVD} \tag{RVD}
\frac{\mu_g(B_R(x))}{\mu_g(B_r(x))} \geq a_1 \left( \frac{R}{r} \right)^n, \quad 0 < r \leq R < \mathrm{inj}(x).
\end{equation}

\item
(Comparability of volumes of balls with different centers) Let us assume that there is a constant $a_2 \geq 1$ such that
\begin{equation} \label{eq:comparability} \tag{Comp}
a_2^{-1} \leq \frac{\mu_g(B_R(x_1))}{\mu_g(B_R(x_2))} \leq a_2, \quad 0 < R < \mathrm{inj}(x_1) \land \mathrm{inj}(x_2).
\end{equation}
\end{itemize}

Let us close this section with the following generalization of Euclidean dyadic cubes that will be used in the decomposition of the contact set and in the Calder\'on–Zygmund technique.

\begin{theorem} [Christ \cite{Chr90}] \label{thm:dyadic_cubes}
There is a countable collection $\mathcal{D} := \lbrace Q_\alpha^j \subset M : j \in \mathbb{Z}, \alpha \in I_j \rbrace$ of open sets and constants $c_1, c_2 > 0$ (with $2c_1 \leq c_2$), and $\delta_0 \in (0,1)$, depending only on $n$, such that
\begin{enumerate}[(i)]
\item
$\mu(M \setminus \cup_\alpha Q_\alpha^j) = 0$ for each $j \in \mathbb{Z}$,
\item
if $i \geq j$, then either $Q_\beta^i \subset Q_\alpha^j$ or $Q_\beta^i \cap Q_\alpha^j = \emptyset$,
\item
for each $(j, \alpha)$ and each $i < j$, there is a unique $\beta$ such that $Q_\alpha^j \subset Q_\beta^i$, 
\item
$\mathrm{diam}(Q_\alpha^j) \leq c_2 \delta_0^j$, and
\item
each $Q_\alpha^j$ contains some ball $B(z_\alpha^j, c_1 \delta_0^j)$.
\end{enumerate}
\end{theorem}

\section{Discrete ABP-type estimates} \label{sec:ABP}

We begin with a discrete version of the ABP-type estimate which will play a key role in the estimates of sub-level sets of $u$ in \Cref{sec:L_eps_estimate}. Cabr\'e suggested in \cite{Cab97} the use of distance squared functions instead of affine functions as touching functions due to the fact that there is no non-constant affine functions in general. This leads us to the smooth map
\begin{equation} \label{eq:phi}
y = \exp_x \nabla (R^2 u)(x).
\end{equation}
That is, if $u$ is a smooth function satisfying $u \geq 0$ in $M \setminus B_{5R}$ and $\inf_{B_{2R}} u \leq 1$, then for any point $y \in B_R$, the minimum of the function $R^2 u + \frac{1}{2} d_y^2$ in $\overline{B}_{5R}$ is achieved at some point $x \in B_{5R}$, leading us to the smooth map \eqref{eq:phi}. For the second order operators, the Jacobian of this smooth map is controlled by the determinant of $D^2 u$, which is in turn controlled by $f$ through the equations. However, since nonlocal operators have order strictly less than two, we cannot go through the determinant of $D^2 u$.

Motivated by the idea of proof of the discrete ABP estimates in \cite{CS09}, we therefore find a small ring around each contact point, in which $u$ stays quadratically close to the envelope. The main difference is that we need to construct the envelope using the distance squared functions instead of affine functions. For each $y \in B_R$, there is a unique paraboloid
\begin{equation*}
P_y(z) = c_y - \frac{1}{2R^2} d_y(z)^2
\end{equation*}
that touches $u$ from below, with a contact point $x \in B_{5R}$. We define the {\it envelope $\Gamma$ of $u$} by
\begin{equation*}
\Gamma(z) = \sup_{y \in B_R} P_y(z),
\end{equation*}
and the {\it contact set} $A = \lbrace x \in B_{5R}: u(x) = \Gamma(x) \rbrace$. In the sequel, let us fix the universal constants
\begin{equation} \label{eq:rho}
\rho_1 = 2\left( \frac{1}{a_1} \right)^{1/n} \lor \frac{1}{\delta_0} > 1 \quad\text{and} \quad \rho_0 < \frac{2c_1}{(3+4/\rho_1)c_2} \delta_0 < 1,
\end{equation}
where $c_1$, $c_2$, and $\delta_0$ are constants, depending only on $n$, in \Cref{thm:dyadic_cubes}, and $a_1$ is the constant in the reverse volume doubling property \eqref{eq:RVD}. This section is devoted to the following nonlocal ABP-type estimate on a Riemannian manifold with nonnegative sectional curvature that generalizes the result in \cite{Cab97}. Recall that $\mathcal{D}$ is a family of dyadic cubes in \Cref{thm:dyadic_cubes}.

\begin{lemma} \label{lem:key}
Let $\sigma_0 \in (0,2)$ and assume $\sigma \in [\sigma_0, 2)$. For $z_0 \in M$, let $K$ be the supremum of the sectional curvatures in $B_{\mathrm{inj}(z_0)}(z_0)$ and let $15R < \mathrm{inj} (z_0) \land \frac{\pi}{\sqrt{K}}$. Let $u \in C^2(B_{5R}(z_0)) \cap L^\infty(M)$ be a function on $M$ satisfying $u \geq 0$ in $M \setminus B_{5R}(z_0)$ and $\inf_{B_{2R}(z_0)} u \leq 1$, and let $\Gamma$ be the envelope of $u$. If $\mathcal{M}^-_{\mathcal{L}_0} u \leq f$ in $B_{5R}(z_0)$, then
\begin{equation} \label{eq:key}
\mu_g(B_R(z_0)) \leq C \sum_{\mathcal{D}_1} \bigg( \Lambda + R^\sigma \max_{\overline{Q}_\alpha^j} f \bigg)_+^n \mu_g(Q_\alpha^j),
\end{equation}
where $\mathcal{D}_1 = \lbrace Q_\alpha^j \rbrace$ is a finite subcollection of $\mathcal{D}$ of dyadic cubes, with $\mathrm{diam}(Q^j_\alpha) \leq \rho_0 \rho_1^{-1/(2-\sigma)} R$, that intersect with the contact set $A$ and satisfy $A \subset \cup_j \overline{Q}_\alpha^j$. The constant $C$ depends only on $n$, $\lambda$, $a_1$, $a_2$, and $\sigma_0$.
\end{lemma}

It is known \cite{GS12} that the estimates \eqref{eq:key} with the Riemann sums in the right hand side replaced by $\Ve \Lambda + f \Ve_{L^n(A)}$ fails to hold even in the case of the Euclidean space. Instead, as in \cite{CS09}, the Riemann sums of $\Lambda + f$ over the set $\cup_j \overline{Q}_\alpha^j$ need to be considered. Thus, we need information not only on the contact set, but also on $\cup_j \overline{Q}_\alpha^j \setminus A$. The map \eqref{eq:phi} is not appropriate as a normal map since $u$ and $\Gamma$ do not coincide outside the contact set. Instead, we will make use of the map $\phi$, which assigns each point $x \in M$ the vertex point $y$ of the paraboloid $P_y$, where $P_y$ is some paraboloid such that $\Gamma(x) = P_y(x)$.  Note that the map $\phi$ may be multivalued since $P_y$ may not be unique.

By using the map $\phi$, we prove the following discrete ABP-type estimates.

\begin{lemma} \label{lem:ABP_discrete}
Assume the same assumptions as in \Cref{lem:key}. There is a finite subcollection $\mathcal{D}_1 \subset \mathcal{D}$ of dyadic cubes $Q_\alpha^j$, with diameters $d_{j,\alpha} \leq \rho_0 \rho_1^{-1/(2-\sigma)} R$, such that the following holds:
\begin{enumerate}[(i)]
\item
Any two different dyadic cubes in $\mathcal{D}_1$ do not intersect.
\item
$A \subset \bigcup_{\mathcal{D}_1} \overline{Q}_\alpha^j$.
\item
$\mu_g(\phi(\overline{Q}_\alpha^j)) \leq C \left(\Lambda + R^\sigma \max_{\overline{Q}_\alpha^j} f \right)_+^n \mu_g(Q_\alpha^j)$.
\item
$\gamma \mu_g(Q_\alpha^j) \leq \mu_g \left( B(z_\alpha^j, (1+4/\rho_1)c_2 \delta_0^j) \cap \lbrace u \leq \Gamma + CR^{-2} (\Lambda + R^\sigma \max_{\overline{Q}_\alpha^j} f)_+ d_{j, \alpha}^2 \rbrace \right)$.
\end{enumerate}
The constants $C > 0$ and $\gamma > 0$ depend only on $n$, $\lambda$, $a_1$, $a_2$, and $\sigma_0$.
\end{lemma}

It is easy to see that \Cref{lem:key} follows from \Cref{lem:ABP_discrete}. Indeed, since we have tested all distance squared function centered on $B_R$, we have $B_R \subset \phi(A)$. Hence, for the family $\mathcal{D}_1$ of dyadic cubes constructed in \Cref{lem:ABP_discrete} we obtain
\begin{equation*}
\mu_g(B_R) \leq \mu_g(\phi(A)) \leq \sum_{\mathcal{D}_1} \mu_g(\phi(\overline{Q}_\alpha^j)) \leq C \sum_{\mathcal{D}_1} \bigg( \Lambda + R^\sigma \max_{\overline{Q}_\alpha^j} f \bigg)_+^n \mu_g(Q_\alpha^j).
\end{equation*}
We postpone the proof of \Cref{lem:ABP_discrete} until the end of this section because we need a series of lemmas in order to prove it.

The next lemma finds a ring around a contact point, where $u$ is quadratically close to the paraboloid in a large portion of the ring. Note that if $x$ is a contact point, then $\Gamma$ is touched by $u$ from above and by some paraboloid $P_y$ from below at $x$, which shows that the paraboloid $P_y$ is uniquely determined and $\Gamma$ is differentiable at $x$. Moreover, in this case, we have
\begin{equation*}
y = \phi(x) = \exp_x \nabla (R^2 \Gamma)(x) = \exp_x \nabla (R^2 u)(x),
\end{equation*}
and hence $\phi(x)$ is also uniquely determined.

\begin{lemma} \label{lem:good_ring}
Assume the same assumptions as in \Cref{lem:key}, and let $r_k = \rho_0 \rho_1^{-1/(2-\sigma)-k} R$. Then, there exists a universal constant $C_0 > 0$, depending only on $n$, $\lambda$, $a_1$, and $\sigma_0$, such that for any $x \in A$ and any $M_0 > 0$, there is an integer $k \geq 0$ such that
\begin{equation} \label{eq:good_ring}
\mu_g(G_k) \leq \frac{C_0}{M_0} \left( \Lambda + R^\sigma f(x) \right)_+ \mu_g( R_k ),
\end{equation}
where $R_k = B(x, r_k) \setminus B(x, r_{k+1})$, $G_k = \lbrace z \in R_k : u(z) > P_y(z) + M_0(r_k/R)^2 \rbrace$, and $y = \phi(x)$.
\end{lemma}

\begin{proof}
Let $x \in A$. By \cite{Xu18}, we have
\begin{equation*}
\mathrm{inj} (x) \geq ( \mathrm{inj} (z_0) \land \mathrm{conj}(x) ) - d(x, z_0).
\end{equation*}
Using $\mathrm{inj} (z_0) > 15R$, $d(x, z_0) < 5R$, and
\begin{equation*}
\mathrm{conj}(x) \geq \frac{\pi}{\sqrt{K}} \land (\mathrm{inj} (z_0) - d(x, z_0)) > 10R,
\end{equation*}
we obtain that
\begin{equation} \label{eq:inj_x}
\mathrm{inj} (x) > 5R.
\end{equation}
Let us compute $\mathcal{M}^-_{\mathcal{L}_0} u(x) = \inf_{L \in \mathcal{L}_0} (I_1 + I_2 + I_3)$, where
\begin{align*}
I_1 &= \int_{B_R(x) \cup B_{5R}(z_0)} \left( u(z) + \frac{1}{2R^2} d_y^2(z) - \left( u(x) + \frac{1}{2R^2} d_y^2(x) \right) \right) \nu_x(z) \, \mathrm{d}V(z), \\
I_2 &= - \int_{B_R(x) \cup B_{5R}(z_0)} \left( \frac{1}{2R^2} d_y^2(z) - \frac{1}{2R^2} d_y^2(x) \right) \nu_x(z) \, \mathrm{d}V(z), \quad\text{and} \\
I_3 &= \int_{M \setminus (B_R(x) \cup B_{5R}(z_0))} (u(z) - u(x)) \nu_x(z) \, \mathrm{d}V(z).
\end{align*}
By the fact \eqref{eq:inj_x}, the symmetry \eqref{eq:symmetry} of density functions $\nu_x(z)$, \Cref{lem:dist_squared}, and \Cref{lem:well-definedness}, we have
\begin{equation*}
\begin{split}
I_2 
=& - \frac{1}{R^2} \int_{B_R(x)} \delta (d_y^2 /2, x, z) \nu_x(z) \, \mathrm{d}V_s(z) \\
&- \frac{1}{2R^2} \int_{B_R(x)} \left( d_y^2(z) - d_y^2(x) \right) \nu_x(z) \, \mathrm{d}V_a(z) \\
&- \frac{1}{2R^2} \int_{B_{5R}(z_0) \setminus B_R(x)} \left( d_y^2(z) - d_y^2(x) \right) \nu_x(z) \, \mathrm{d}V(z) \\
\geq& -C \Lambda R^{-\sigma},
\end{split}
\end{equation*}
where $C = C(n, \sigma_0)$ is some universal constant.

On the other hand, we know that $u(x) \leq u(x) + \frac{1}{2R^2} d_y^2(x) \leq \inf_{B_{2R}} (u+\frac{1}{2R^2} d_y^2) \leq 11/2 < 6$. This fact together with the assumption that $u \geq 0$ in $M \setminus B_{5R}(z_0)$ leads us to
\begin{equation*}
I_3 \geq - \int_{M \setminus B_R(x)} \frac{6\Lambda (2-\sigma)}{\mu_g(B(x, d_x(z))) d_x(z)^\sigma} \, \mathrm{d}V(z) \geq -C\Lambda R^{-\sigma}
\end{equation*}
by the similar argument.

Let us now estimate $I_1$. Since the contact point $x$ minimizes the function $u + \frac{1}{2R^2} d_y^2$, the integrand in $I_1$ is nonnegative. Thus, we have
\begin{equation*}
I_1 \geq \lambda (2-\sigma) \sum_{k=0}^\infty \int_{G_k} \frac{u(z) + \frac{1}{2R^2} d_y^2(z) - \left( u(x) + \frac{1}{2R^2} d_y^2(x) \right)}{\mu_g(B(x, d_x(z))) d_x(z)^\sigma} \, \mathrm{d}V(z).
\end{equation*}
Let us assume to the contrary that \eqref{eq:good_ring} does not hold for all $k \geq 0$, that is, 
\begin{equation} \label{eq:hypothesis}
\mu_g( G_k ) > \frac{C_0}{M_0} (\Lambda + R^\sigma f(x))_+ \mu_g( R_k ) \quad\text{for all}~ k \geq 0,
\end{equation}
for some $C_0 > 0$ that will be chosen at the end of the proof. If $z \in G_k$, then $u(z) + \frac{1}{2R^2} d_y^2(z) - \left( u(x) + \frac{1}{2R^2} d_y^2(x) \right) \geq M_0 (r_k/R)^2$. Thus, using \eqref{eq:hypothesis} and the reverse volume doubling property \eqref{eq:RVD}, we obtain the lower bound of $I_1$ as
\begin{align*}
I_1 
&\geq \lambda (2-\sigma) C_0 \sum_{k=0}^\infty \frac{(r_k / R)^2}{\mu_g(B(x, r_k)) r_k^\sigma} \left( \Lambda + R^\sigma f(x) \right)_+ \mu_g(R_k) \\
&= \lambda \frac{2-\sigma}{\rho_1} C_0 \sum_{k=0}^\infty \rho_0^{2-\sigma} \rho_1^{-k(2-\sigma)} R^{-\sigma} \left( 1- \frac{\mu_g(B(x, r_{k+1}))}{\mu_g(B(x, r_k))} \right) \left( \Lambda + R^\sigma f(x) \right)_+ \\
&\geq \lambda \frac{\rho_0^2}{\rho_1} C_0 \frac{2-\sigma}{1-\rho_1^{-(2-\sigma)}} \left( 1- \frac{1}{a_1 \rho_1^n} \right) \left( \Lambda R^{-\sigma} + f(x) \right)_+.
\end{align*}
Recalling \eqref{eq:rho} and observing that the function $t/(1-\rho_1^{-t})$ is bounded away from 0 in $[0, 2]$, we arrive at
\begin{equation*}
I_1 \geq c_1 C_0 \left( \Lambda R^{-\sigma} + f(x) \right)_+,
\end{equation*}
where $c_1 = c_1(n, \lambda, a_1) > 0$.

We have obtained that
\begin{equation*}
f(x) \geq \mathcal{M}^-_{\mathcal{L}_0} u(x) \geq \inf_{L \in \mathcal{L}_0} (I_1 + I_2 + I_3) \geq c_1 C_0 \left( \Lambda R^{-\sigma} + f(x) \right)_+ - C\Lambda R^{-\sigma}.
\end{equation*}
Therefore, by taking $C_0$ sufficiently large, we arrive at a contradiction.
\end{proof}

The next lemma shows that the function $\Gamma - P_y$ is $-R^{-2}$-convex in the sense of second order incremental quotients.

\begin{lemma} \label{lem:convexity}
Let $x \in A$, $y \in \phi(x)$, $K = K_{\mathrm{max}}(B_{\mathrm{inj}(y)}(y))$, and let $\rho > 0$ satisfy \eqref{eq:rho_K}. For $z \in B_\rho(y)$, let $\xi \in T_z M$ be such that $\exp_z(s \xi) \in B_\rho(y)$ for all $s \in (-1, 1)$. Then,
\begin{equation} \label{eq:convexity}
(\Gamma - P_y)(z) \leq (1-t) (\Gamma - P_y)(z_1) + t(\Gamma - P_y)(z_2) + \frac{1}{2R^2} t(1-t) \ve \xi \ve_g^2
\end{equation}
for all $t \in (0,1)$, where $z_1 = \exp_z(t \xi)$ and $z_2 = \exp_z((1-t)(-\xi))$.
\end{lemma}

\begin{proof}
By the definition of $\Gamma$, there is a paraboloid $P_\ast := P_{y_\ast}$ with some point $y_\ast \in B_R$, such that $\Gamma(z) = P_\ast(z)$. Then we have
\begin{equation*}
(\Gamma - P_\ast)(z) = 0 \leq (1-t)(\Gamma - P_\ast)(z_1) + t (\Gamma - P_\ast)(z_2),
\end{equation*}
and hence
\begin{align} \label{eq:convexity_ast}
\begin{split}
(\Gamma - P_y)(z) 
\leq&~ (1-t) (\Gamma - P_y)(z_1) + t(\Gamma - P_y)(z_2) \\
&~ - \left( (1-t) P_\ast(z_1) + t P_\ast(z_2) - P_\ast(z) \right) \\
&~ + \left( (1-t) P_y(z_1) + t P_y(z_2) - P_y(z) \right).
\end{split}
\end{align}
Using \Cref{lem:dist_squared}, we obtain
\begin{align} \label{eq:P_ast}
\begin{split}
- \left( (1-t) P_\ast(z_1) + t P_\ast(z_2) - P_\ast(z) \right) 
&= \frac{1}{2R^2} \left( (1-t) d_{y_\ast}^2(z_1) + t d_{y_\ast}^2 (z_2) - d_{y_\ast}^2(z) \right) \\
&\leq \frac{1}{2R^2} t(1-t) \ve \xi \ve_g^2
\end{split}
\end{align}
and
\begin{equation} \label{eq:P_y}
(1-t) P_y(z_1) + t P_y(z_2) - P_y(z) = - \frac{1}{2R^2} \left( (1-t) d_y^2(z_1) + t d_y^2 (z_2) - d_y^2(z) \right) \leq 0.
\end{equation}
Therefore, \eqref{eq:convexity} follows from \eqref{eq:convexity_ast}, \eqref{eq:P_ast}, and \eqref{eq:P_y}.
\end{proof}

By means of \Cref{lem:good_ring} and \Cref{lem:convexity}, we will show that in a small ball near a contact point the envelope is captured by two paraboloids that are quadratically close to each other. Recall that the convex envelope constructed by affine functions in the case of Euclidean spaces \cite{CS09} is captured by two affine planes. The idea in \cite{CS09} is to carry information from the ``good ring" to the ball enclosed by the ring, by using the convexity of the function $\Gamma$. In our setting, we use $-R^{-2}$-convexity of the function $\Gamma - P_y$ instead.

\begin{lemma} \label{lem:geometric}
Assume the same assumptions as in \Cref{lem:key}. Let $x \in A$, $y = \phi(x)$, and let $r = r_k$ be the radius in \Cref{lem:good_ring}. There is a small constant $\varepsilon_0 = \varepsilon_0(n) \in (0,1)$ such that if
\begin{equation} \label{eq:geo_assumption}
\mu_g(\lbrace z \in B_r(x) \setminus B_{r/2}(x): \Gamma(z) > P_y(z) + h \rbrace) \leq \varepsilon_0 \mu_g(B_r(x) \setminus B_{r/2}(x)),
\end{equation}
then
\begin{equation*}
\Gamma(z) \leq P_y(z) + h + \frac{1}{2} \left( \frac{r}{R} \right)^2
\end{equation*}
for all $z \in B_{r/2}(x)$.
\end{lemma}

\begin{proof}
Let us fix $z \in B_{r/2}(x)$ and claim that there are two points $w_1, w_2 \in B_r(x) \setminus B_{r/2}(x)$ such that three points $w_1$, $z$, and $w_2$ are joined by a geodesic, and that
\begin{equation} \label{eq:w_i}
\Gamma(w_i) \leq P_y(w_i) + h, \quad i = 1, 2.
\end{equation}
Once we find such points, we may write $w_1 = \exp_z(t \xi)$ and $w_2 = \exp_z((1-t)(-\xi))$ for some $\xi \in T_z M$ with $\ve \xi \ve_g$ being the length of the line segment between $w_1$ and $w_2$, and $t \in (0, 1)$. Then, we have $\ve \xi \ve_g \leq 2r$ and $t(1-t) \leq 1/4$. Thus, by \Cref{lem:convexity} and \eqref{eq:w_i}, we obtain that
\begin{equation*}
(\Gamma - P_y)(z) \leq h + \frac{1}{2R^2} t(1-t) \ve \xi \ve_g^2 \leq h + \frac{1}{2} \left( \frac{r}{R} \right)^2,
\end{equation*}
finishing the proof.

To prove the claim, we first extend the line segment between $x$ and $z$ in both directions to find two points $z_1$ and $z_2$ on $\partial B_{3r/4}(x)$. We call the farther one from $z$ as $z_1$ and the closer one from $z$ as $z_2$. Let $D = \lbrace z: \Gamma(z) \leq P_y(z) + h \rbrace$, then it follows from \eqref{eq:geo_assumption} and \eqref{eq:VD} that
\begin{align*}
\mu_g(B_{r/8}(z_1) \cap D^c) 
&\leq \mu_g((B_r(x) \setminus B_{r/2}(x)) \cap D^c) \\
&\leq \varepsilon_0 \mu_g(B_r(x) \setminus B_{r/2}(x)) \\
&\leq \varepsilon_0 \mu_g(B_{2r}(z_1)) \leq \varepsilon_0 16^n \mu_g(B_{r/8}(z_1)).
\end{align*}
Assuming $\varepsilon_0 < 2^{-4n-1}$, we obtain
\begin{equation} \label{eq:half}
\mu_g(D \cap B_{r/8}(z_1)) \geq \frac{1}{2} \mu_g(B_{r/8}(z_1)).
\end{equation}
Let us write $d = d(z, x) \in [0, r/2)$ and define a map $F = \exp_z \circ T \circ \exp_z^{-1} : B_{\mathrm{inj}(z)}(z) \to M$, where $T: T_z M \to T_z M$ is a linear map given by
\begin{equation*}
T(\xi) = - \frac{\frac{3r}{4} - d}{\frac{3r}{4} + d} \, \xi.
\end{equation*}
By \cite{Xu18}, we know that $\mathrm{inj}(z) \geq \left( \mathrm{inj}(z_0) \land \mathrm{conj} (z) \right) - d(z, z_0)$. It follows from $\mathrm{inj}(z_0) > 15R$, $d(z, z_0) \leq d(z, x) + d(x, z_0) < r/2 + 5R$, and
\begin{equation*}
\mathrm{conj}(z) \geq \frac{\pi}{\sqrt{K}} \land \left( \mathrm{inj}(z_0) - d(z, z_0) \right) \geq 10R - r/2,
\end{equation*}
that $\mathrm{inj}(z) \geq 5R - r$. Let us recall that we have $r = r_k = \rho_0 \rho_1^{-1/(2-\sigma)-k} R < R$ from the choice \eqref{eq:rho}. Thus, we obtain $\mathrm{inj}(z) > 4r$, and hence $F$ is well-defined in $B_{r/8}(z_1) \subset B_{\mathrm{inj}(z)}(z)$. 

What we only need to show is that $(B_r(x) \setminus B_{r/2}(x)) \cap D \cap F(D \cap B_{r/8}(z_1))$ is not empty. Let us assume to the contrary that
\begin{equation} \label{eq:contradiction}
F(D \cap B_{r/8}(z_1)) \subset (B_r(x) \setminus B_{r/2}(x)) \cap D^c
\end{equation}
and find a contradiction by estimating the volume change by $F$. Let $D_1 = \exp_z^{-1} (D \cap B_{r/8}(z_1))$, $D_2 = T(D_1)$, and $D_3 = \exp_z(D_2)$. Due to the lower bound of curvature, we estimate
\begin{equation} \label{eq:D_1}
\mu_g(D \cap B_{r/8}(z_1)) = \iint_{D_1} \det (D \exp_z)(tv) t^{n-1} \, \d v \, \d t \leq \iint_{D_1} t^{n-1} \, \d v \, \d t = \ve D_1 \ve.
\end{equation}
by means of the polar coordinates (see, for instance, \cite{Jos17}). Since $d \in [0, r/2)$, we have
\begin{equation} \label{eq:D_2}
|D_2| = \left( \frac{\frac{3r}{4}-d}{\frac{3r}{4}+d} \right) ^n |D_1| \geq 5^{-n} \ve D_1 \ve.
\end{equation}
Moreover, since the curvature in $D_3$ is bounded from above by $K_{\mathrm{max}}(D_3)$, which is less than or equal to $K = K_{\mathrm{max}}(B_{\mathrm{inj}(z_0)}(z_0))$, we obtain
\begin{equation*}
\mu_g(D_3) = \iint_{D_2} \det(D \exp_z)(tv) t^{n-1} \, \d v \, \d t \geq \iint_{D_2} \left( \frac{\sin ( \sqrt{K} t )}{\sqrt{K} t} \right)^{n-1} t^{n-1} \, \d v \, \d t.
\end{equation*}
Note that the function $\sin(\sqrt{K} t)/(\sqrt{K}t)$ is nonnegative and decreasing in $(0, \frac{\pi}{\sqrt{K}}]$.
If $(t, v) \in D_2$, then
\begin{equation*}
t \leq \frac{\frac{3r}{4} - d}{\frac{3r}{4} + d} \left( \frac{7r}{8} + d \right) \leq \frac{3r}{2} < \frac{3R}{2} < \frac{\pi}{10 \sqrt{K}},
\end{equation*}
and hence
\begin{equation} \label{eq:D_3}
\mu_g(D_3) \geq \left( \frac{\sin \frac{\pi}{10}}{\frac{\pi}{10}}\right)^{n-1} \ve D_2 \ve.
\end{equation}
Combining \eqref{eq:half} and \eqref{eq:D_1}–\eqref{eq:D_3}, we have $\mu_g(B_{r/8}(z_1)) \leq C(n) \mu_g(D_3)$. Moreover, by using \eqref{eq:contradiction}, \eqref{eq:geo_assumption}, and \eqref{eq:VD}, we obtain
\begin{align*}
\mu_g(B_{r/8}(z_1)) 
&\leq C \mu_g((B_r(x) \setminus B_{r/2}(x)) \cap D^c) \\
&\leq C \varepsilon_0 \mu_g(B_r(x) \setminus B_{r/2}(x)) \\
&\leq C \varepsilon_0 \mu_g(B_{2r}(z_1)) \leq 16^n C \varepsilon_0 \mu_g(B_{r/8}(z_1)).
\end{align*}
Therefore, we arrive at a contradiction by taking $\varepsilon_0 < 16^{-n} C^{-1}$.
\end{proof}

The flatness of $\Gamma$ in a small region, obtained in \Cref{lem:geometric}, allows us to control the gradient of $\Gamma$ in a smaller region, where the gradient of $\Gamma$ is understood as the gradient of touching paraboloid. This is done by estimating the image of the map $\phi$.

\begin{lemma} \label{lem:gradient_map}
Assume the same assumptions as in \Cref{lem:good_ring}, and let $\varepsilon_0$ be the constant in \Cref{lem:geometric}. For any $x \in A$ there is an $r = r_k \leq r_0$ such that
\begin{equation} \label{eq:ring_eps}
\mu_g \left( \left\lbrace z \in R_k: u(z) > P_y(z) + C\left( \Lambda + R^\sigma f(x) \right)_+ (r_k/R)^2 \right\rbrace \right) \leq \varepsilon_0 \mu_g(R_k)
\end{equation}
and
\begin{equation} \label{eq:gradient_map}
\mu_g \left(\phi \left(\overline{B_{r/4}(x)} \right) \right) \leq C \left( \Lambda + R^\sigma f(x) \right)_+^n \mu_g(B_{r/4}(x)),
\end{equation}
where $C > 0$ is a universal constant depending only on $n$, $\lambda$, $a_1$, $a_2$, and $\sigma_0$.
\end{lemma}

\begin{proof}
Let $x \in A$ and $y = \phi(x)$. By applying \Cref{lem:good_ring} to $u$ with $M_0 = \frac{C_0}{\varepsilon_0} \left( \Lambda + R^\sigma f(x) \right)_+$, we find $r = r_k \leq r_0$ such that \eqref{eq:ring_eps} holds. Moreover, since $\Gamma \leq u$, we have
\begin{equation*}
\mu_g \left( \left\lbrace z \in R_k : \Gamma (z) > P_y(z) + C \left( \Lambda + R^\sigma f(x) \right)_+ (r_k/R)^2 \right\rbrace \right) \leq \varepsilon_0 \mu_g(R_k).
\end{equation*}
Thus, \Cref{lem:geometric} shows that
\begin{equation} \label{eq:flatness}
P_y(z) \leq \Gamma(z) \leq P_y(z) + C \left( \Lambda + R^\sigma f(x) \right)_+ \left( \frac{r}{R} \right)^2
\end{equation}
for all $z \in B_{r/2}(x)$.

We first claim that there is a constant $C_1 > 0$ such that
\begin{equation} \label{eq:grad}
\ve \nabla P_{y_\ast} (z) - \nabla P_y (z) \ve_{g(z)} \leq \frac{C_1}{R^2} \left( \Lambda + R^\sigma f(x) \right)_+ r
\end{equation}
for all $z \in B_{r/4}(x)$ and $y_\ast \in \phi(z)$. It is enough to show that
\begin{equation} \label{eq:directional_grad}
\left\ve \left. \frac{\d}{\d t} \right\ve_{t=0} \left( P_{y_\ast} - P_y \right)(\gamma(t)) \right\ve \leq \frac{C_1}{R^2} \left( \Lambda + R^\sigma f(x) \right)_+ r
\end{equation}
for all geodesics $\gamma$, with unit speed, starting from $\gamma(0) = z$. Suppose that there is a geodesic $\gamma$ such that \eqref{eq:directional_grad} does not hold. We may assume that
\begin{equation*}
\frac{C_1}{R^2} \left( \Lambda + R^\sigma f(x) \right)_+ r \leq \left. \frac{\d}{\d t} \right\ve_{t=0} \left( P_{y_\ast} - P_y \right)(\gamma(t)),
\end{equation*}
by considering $\tilde{\gamma}(t) = \gamma(-t)$ instead of $\gamma(t)$ if necessary. Let $\varepsilon > 0$, then there is $\delta > 0$ such that if $\ve t \ve < \delta$, we have
\begin{equation} \label{eq:t}
\frac{C_1}{R^2} \left( \Lambda + R^\sigma f(x) \right)_+ r - \varepsilon \leq \frac{(P_{y_\ast} - P_y)(\gamma(t)) - (P_{y_\ast} - P_y)(\gamma(0))}{t} \leq \frac{h(t) - h(0)}{t},
\end{equation}
where $h(t) = (\Gamma - P_y)(\gamma(t))$. Let $T > 0$ be the first time when $\gamma$ reaches the boundary of $B_{3r/8}(x)$, namely, $\gamma(T) \in \partial B_{3r/8}(x)$. Let $N$ be the least integer not smaller than $T/\delta$, and let $0 = t_0 < t_1 < \cdots < t_N = T$ be equally distributed times. Then we have $t_{i+1} - t_i = T/N \leq \delta$. We observe that \Cref{lem:convexity} shows
\begin{equation} \label{eq:t_i}
\frac{h(t_i) - h(t_{i-1})}{t_i - t_{i-1}} \leq \frac{h(t_{i+1}) - h(t_i)}{t_{i+1} - t_i} + \frac{1}{2R^2} (t_{i+1} - t_{i-1}), \quad i = 1, 2, \cdots, N-1.
\end{equation}
Thus, it follows from \eqref{eq:t} and \eqref{eq:t_i} that
\begin{align*}
\frac{C_1}{R^2} \left( \Lambda + R^\sigma f(x) \right)_+ r - \varepsilon 
&\leq \frac{h(t_1) - h(t_0)}{T/N} \leq \frac{h(t_2) - h(t_1)}{T/N} + \frac{1}{2R^2} \frac{2T}{N} \\
&\leq \cdots \leq \frac{h(t_N) - h(t_{N-1})}{T/N} + \frac{1}{2R^2} \frac{2T}{N} (N-1).
\end{align*}
Therefore, we obtain that
\begin{equation*}
N \left( \frac{C_1}{R^2} \left( \Lambda + R^\sigma f(x) \right)_+ r - \varepsilon \right) \leq \frac{h(t_N) - h(t_0)}{T/N} + \frac{1}{2R^2} \frac{2T}{N} \frac{N(N-1)}{2}.
\end{equation*}
Since $\gamma$ has a unit speed, we have $r/8 < T <r$, and hence
\begin{equation*}
\frac{C_1}{R^2} \left( \Lambda + R^\sigma f(x) \right)_+ r - \varepsilon \leq \frac{(\Gamma - P_y)(\gamma(T)) - (\Gamma - P_y)(z)}{r/8} + \frac{r}{2R^2} \leq \frac{(\Gamma - P_y)(\gamma(T))}{r/8} + \frac{r}{2R^2}.
\end{equation*}
Recalling that $\varepsilon$ was arbitrary, we have
\begin{equation} \label{eq:not_flat}
\frac{C_1}{8} \left( \Lambda + R^\sigma f(x) \right)_+ \left( \frac{r}{R} \right)^2 - \frac{r^2}{16R^2} \leq (\Gamma - P_y) (\gamma(T)).
\end{equation}
Since $\gamma(T) \in \partial B_{3r/8}(x) \subset B_{r/2}(x)$, the inequality \eqref{eq:not_flat} with sufficiently large constant $C_1 > 0$ contradicts to \eqref{eq:flatness}. Therefore, we have proved the claim \eqref{eq:grad}.

Let us next prove \eqref{eq:gradient_map} using \eqref{eq:grad}. It is enough to show that
\begin{equation} \label{eq:image_phi}
\phi \left( \overline{B(x, r/4)} \right) \subset B \Big( \phi(x), C(\Lambda + R^\sigma f(x))_+ r \Big).
\end{equation}
Indeed, once \eqref{eq:image_phi} is proved, then by \eqref{eq:comparability} and \eqref{eq:VD} (or \eqref{eq:RVD}) we have
\begin{equation*}
\mu_g \left( \phi \left( \overline{B(x, r/4)} \right) \right) \leq \mu_g \Big( B( \phi(x), C(\Lambda + R^\sigma f(x))_+r ) \Big) \leq C \left( \Lambda + R^\sigma f(x) \right)_+^n \mu_g(B(x, r/4)).
\end{equation*}
To verify \eqref{eq:image_phi}, let us fix $z \in \overline{B(x, r/4)}$ and $y_\ast \in \phi(z)$. Then we know from \eqref{eq:grad} that
\begin{equation*}
\left\ve \exp_z^{-1} y_\ast - \exp_z^{-1} y \right\ve_{g(z)} = R^2 \ve \nabla P_{y_\ast} (z) - \nabla P_y (z) \ve_{g(z)} \leq C_1 \left( \Lambda + R^\sigma f(x) \right)_+ r.
\end{equation*}
Thus, it only remains to show that
\begin{equation} \label{eq:dist_vertex}
d(y_\ast, y) \leq \left\ve \exp_z^{-1} y_\ast - \exp_z^{-1} y \right\ve_{g(z)}.
\end{equation}
Let $\xi_1 = \exp_z^{-1} y_\ast$ and $\xi_2 = \exp_z^{-1} y$. Let us consider a family of geodesics
\begin{equation*}
\gamma(s, t) = \exp_z (t(\xi_1 + s(\xi_2 - \xi_1))),
\end{equation*}
and the Jacobi field $J$ along $\gamma$. Then, by \cite[Equation (1.9)]{Gig12} (or see, e.g. \cite{Jos17}), we have
\begin{equation*}
\ve J(1) \ve_{g(y_\ast)} \leq \ve J'(0) \ve_{g(z)} = \ve \xi_2 - \xi_1 \ve_{g(z)}.
\end{equation*}
Therefore, \eqref{eq:dist_vertex} follows by considering the curve $s \mapsto \gamma(s, 1)$ and observing that
\begin{equation*}
d(y_\ast, y) \leq \int_0^1 \ve \gamma'(s, 1) \ve_{g(y_\ast)} \, \d s \leq \int_0^1 \ve \xi_2 - \xi_1 \ve_{g(z)} \, \d s = \ve \xi_2 - \xi_1 \ve_{g(z)}.
\end{equation*}
We have proved \eqref{eq:dist_vertex}, from which we deduce \eqref{eq:image_phi}.
\end{proof}

We are now ready to prove \Cref{lem:ABP_discrete} by using the previous lemmas and the dyadic cubes in \Cref{thm:dyadic_cubes}.

\begin{proof} [Proof of \Cref{lem:ABP_discrete}]
In order to construct such a family, we are going to use \Cref{thm:dyadic_cubes}. Let us first fix the smallest integer $N \in \mathbb{Z}$ such that $c_2 \delta_0^N \leq r_0$. Then there are finitely many dyadic cubes $Q_\alpha^N$ of generation $N$ such that $\overline{Q}_\alpha^N \cap A \neq \emptyset$ and $A \subset \cup_\alpha \overline{Q}_\alpha^N$. Whenever a dyadic cube $Q_\alpha^N$ does not satisfy (iii) and (iv), we consider its successors $Q_\beta^{N+1} \subset Q_\alpha^N$ instead of $Q_\alpha^N$. Among these successors of $N+1$ generation, we only keep those whose closures intersect $A$, and discard the rest. We iterate the process in the same way. The only part we need to prove is that the process finishes in a finite number of steps.

Assume that the process produces an infinite sequence of nested dyadic cubes $\lbrace Q_\alpha^j \rbrace_{j=N}^\infty$ with $\alpha = \alpha_j \in I_j$. Then the intersection of their closures is some point $x_0$ which is contained in the contact set $A$. By \Cref{lem:gradient_map}, we there is an $r = r_k \leq r_0$ such that \eqref{eq:ring_eps} and \eqref{eq:gradient_map} hold. The condition $\rho_1 \geq 1/\delta_0$ in \eqref{eq:rho} allows us to find $j \geq N$ satisfying $\rho_1 r/4 \leq c_2 \delta_0^j < r/4$. Then it follows from \Cref{thm:dyadic_cubes} (iv) that
\begin{equation} \label{eq:Q_in_B}
\overline{Q}_\alpha^j \subset B_{r/4}(x_0).
\end{equation}
Moreover, by \Cref{thm:dyadic_cubes} (v), $Q_\alpha^j$ contains some ball $B(z_\alpha^j, c_1 \delta_0^j)$. If $z \in B_r(x_0)$, then $d(z, z_\alpha^j) \leq d(z, x_0) + d(x_0, z_\alpha^j) < (1+4/\rho_1)c_2 \delta_0^j$, which shows that
\begin{equation} \label{eq:B_in_Q}
B_r(x_0) \subset B(z_\alpha^j, (1+4/\rho_1)c_2 \delta_0^j).
\end{equation}
Therefore, it follows from \eqref{eq:gradient_map}, \eqref{eq:Q_in_B}, \eqref{eq:comparability}, and the volume doubling property \eqref{eq:VD} that
\begin{align*}
\mu_g ( \phi ( \overline{Q}_\alpha^j )) \leq \mu_g(\phi(\overline{B_{r/4}(x_0)})) 
&\leq C \left( \Lambda + R^\sigma f(x_0) \right)_+^n \mu_g(B_{r/4}(x_0)) \\
&\leq C \bigg( \Lambda + R^\sigma \max_{\overline{Q}_\alpha^j} f \bigg)_+^n \mu_g(Q_\alpha^j).
\end{align*}
Furthermore, \eqref{eq:ring_eps}, \eqref{eq:comparability}, \eqref{eq:VD}, and \eqref{eq:RVD} show that
\begin{align*}
&\mu_g \bigg( B(z_\alpha^j, (1+4/\rho_1)c_2 \delta_0^j) \cap \bigg\lbrace u \leq \Gamma + CR^{-2} \bigg( \Lambda + R^\sigma \max_{\overline{Q}_\alpha^j} f \bigg)_+ d_{j, \alpha}^2 \bigg\rbrace \bigg) \\
&\geq \mu_g \bigg( B_r(x_0) \cap \bigg\lbrace u \leq P_y + C \bigg( \Lambda + R^\sigma \max_{\overline{Q}_\alpha^j} f \bigg)_+ \left( \frac{r}{R} \right)^2 \bigg\rbrace \bigg) \\
&\geq (1-\varepsilon) \mu_g(R_k(x_0)) \geq \gamma \mu_g(Q_\alpha^j)
\end{align*}
for some universal constant $\gamma > 0$. We have shown that $Q_\alpha^j$ satisfy (iii) and (iv), which yields a contradiction. Therefore, the process must stop in a finite number of steps.
\end{proof}

\section{A barrier function} \label{sec:barrier}

In this section we construct a barrier function, which is one of the key ingredients for the Krylov–Safonov Harnack inequality. We use the function of the form $(d^2_{z_0}(\cdot))^{-\alpha}$ with $\alpha > 0$ large, which have been used as a barrier function both for nonlocal operators in the Euclidean spaces \cite{CS09} and local operators on Riemannian manifolds \cite{Cab97}. In \cite{Cab97}, the Hessian bound of distance squared function at one point is enough to evaluate the operator's value. However, for nonlocal operators on Riemannian manifolds, the curvatures near the given point have to be taken into account to evaluate the operator. In order to make the universal constants independent of the curvatures, we need to look at a small region.

\begin{lemma} \label{lem:barrier1}
For $z_0 \in M$, let $K$ be the supremum of the sectional curvatures in $B_{\mathrm{inj}(z_0)}(z_0)$ and let $15R < \mathrm{inj}(z_0) \land \frac{\pi}{\sqrt{K}}$. There are universal constants $\alpha > 0$ and $\sigma_0 \in (0,2)$, depending only on $n$, $\lambda$, and $\Lambda$, such that the function
\begin{equation*}
v(x) = \max \bigg\lbrace - \left( \frac{\rho_0}{20} \right)^{-2\alpha}, - \left( \frac{d_{z_0}(x)}{5R} \right)^{-2\alpha} \bigg\rbrace
\end{equation*}
is a supersolution to
\begin{equation*}
R^\sigma \mathcal{M}^+_{\mathcal{L}_0}v(x) + \Lambda \leq 0,
\end{equation*}
for every $\sigma_0 \leq \sigma < 2$ and $x \in B_{5R}(z_0) \setminus \overline{B}_{\rho_0 R}(z_0)$.
\end{lemma}

\begin{proof}
Fix $x$ and let $R_0 := d_{z_0}(x) \in (\rho_0 R, 5R)$. Let us consider normal coordinates centered at $x$, then by \eqref{eq:inj_x} the point $z_0$ is included in the normal coordinates. Thus, we may assume that $\exp_x^{-1} z_0 = R_0 e_1$. Let $\xi := \exp_x^{-1} z$. By the Toponogov's triangle comparison (see, e.g., \cite{Pet16}), we have $d_{z_0}(z) \leq \ve R_0 e_1 - \xi \ve_\delta$ and $d_{z_0}(\mathcal{T}_x(z)) \leq \ve R_0 e_1 + \xi \ve_\delta$, and hence,
\begin{equation*}
\delta(v, x, z) \leq - \left( \frac{1}{5R} \right)^{-2\alpha} \left( \ve R_0 e_1 + \xi \ve_\delta^{-2\alpha} + \ve R_0 e_1 - \xi \ve_\delta^{-2\alpha} - 2 R_0^{-2\alpha} \right)
\end{equation*}
for $d(z, x) \leq R_0/2$. As $T_x M$ being identified as $\mathbb{R}^n$, a simple algebraic inequality shows that
\begin{equation} \label{eq:second_difference}
\delta (v, x, z) \leq 2\alpha \left( \frac{R_0}{5R} \right)^{-2\alpha} \left( \frac{\ve \xi \ve^2}{R_0^2} - (2\alpha+2) \frac{\xi_1^2}{R_0^2} + (2\alpha+2)(\alpha+2) \frac{\xi_1^2 \ve \xi \ve^2}{R_0^4} \right)
\end{equation}
as in the proof of \cite[Lemma 9.1]{CS09}.

Let us take $\alpha = \alpha(n, \lambda, \Lambda) > 0$ sufficiently large so that
\begin{equation} \label{eq:alpha}
\lambda (2\alpha+2) \left( \frac{3}{\pi} \right)^{n-1} \int_{\partial B_1} v_1^2 \, \d v - \Lambda \ve \partial B_1 \ve > C_1 \Lambda,
\end{equation}
for some universal constant $C_1 > 0$ to be determined later, where $\d v$ is the usual spherical measure on $\partial B_1$. Then,
\begin{align} \label{eq:I}
\begin{split}
R^\sigma \mathcal{M}^+_{\mathcal{L}_0} v (x)
\leq ~ & (2-\sigma) R^\sigma \int_{B(0, R_0/2)} \frac{\Lambda \delta(v, x, z)_+ - \lambda \delta(v, x, z)_-}{\mu_{g^\ast} (B(0, \ve \xi \ve)) \ve \xi \ve^\sigma} \, \d V_s^\ast (\xi) \\
&+ R^\sigma \sup_{L \in \mathcal{L}_0} \int_{B(x, R_0/2)} (v(z) - v(x)) \nu_x(z) \, \d V_a (z) \\
&+ R^\sigma \sup_{L \in \mathcal{L}_0} \int_{M \setminus B(x, R_0/2)} (v(z) - v(x)) \nu_x(z) \, \d V(z) =: I_1 + I_2 + I_3.
\end{split}
\end{align}
For $I_1$, we use \eqref{eq:second_difference} to obtain
\begin{equation} \label{eq:I_1}
I_1 \leq 2\alpha (2-\sigma) \frac{R^\sigma}{R_0^2} \left( \frac{5R}{R_0} \right)^{2\alpha} \int_{B_{R_0/2}} \frac{\Lambda \ve \xi \ve^2 - \lambda(2\alpha+2) \xi_1^2 + \Lambda (2\alpha+2) (\alpha+2) \xi_1^2 \ve \xi \ve^2 R_0^{-2}}{\mu_{g^\ast}(B(0, \ve \xi \ve)) \ve \xi \ve^\sigma} \, \d V_s^\ast (\xi).
\end{equation}
Since the sectional curvatures on $B(x, R_0/2)$ are bounded by $K$ from above and 0 from below, we have
\begin{align*}
I_{1,1}
:=& \int_{B_{R_0/2}} \frac{\Lambda \ve \xi \ve^2 - \lambda(2\alpha+2) \xi_1^2}{\mu_{g^\ast}(B(0, \ve \xi \ve)) \ve \xi \ve^\sigma} \, \d V_s^\ast (\xi) \\
=& \int_0^{R_0/2} \int_{\partial B_1} \frac{\Lambda - \lambda(2\alpha+2) v_1^2}{\mu_{g^\ast}(B(0, t)) t^\sigma} \det(D \exp_x)(tv) t^{n+1} \, \d v \, \d t \\
\leq& \int_0^{R_0/2} \int_{\partial B_1} \left( \Lambda - \lambda(2\alpha+2) \bigg( \frac{\sin(\sqrt{K} t)}{\sqrt{K} t} \bigg)^{n-1} v_1^2 \right) \d v \frac{t^{n+1}}{\mu_{g^\ast}(B(0, t)) t^\sigma} \, \d t.
\end{align*}
As in the proof of \Cref{lem:geometric}, we observe from $t \leq R_0/2 < \frac{5}{2} R < \frac{\pi}{6\sqrt{K}}$ that
\begin{equation} \label{eq:sint_t}
\frac{\sin(\sqrt{K}t)}{\sqrt{K} t} \geq \frac{\sin (\pi/6)}{\pi/6} = \frac{3}{\pi}.
\end{equation}
Thus, \eqref{eq:alpha}, \eqref{eq:sint_t}, and the Gromov's theorem yield that
\begin{equation} \label{eq:I_1,1}
I_{1,1} \leq -C_1 \Lambda \int_0^{R_0/2} \frac{t^{n+1}}{\ve B_t \ve t^\sigma} \, \d t = - \frac{C_1 \Lambda}{(2-\sigma) \omega_n} (R_0/2)^{2-\sigma}.
\end{equation}
Similarly, we obtain that
\begin{align} \label{eq:I_1,2}
\begin{split}
I_{1,2}
:=&~ \int_{B_{R_0/2}} \frac{\Lambda(2\alpha+2)(\alpha+2) \xi_1^2 \ve \xi \ve^2 R_0^{-2}}{\mu_{g^\ast}(B(0, \ve \xi \ve)) \ve \xi \ve^\sigma} \, \d V_s^\ast (\xi) \\
\leq&~ C \frac{\Lambda(2\alpha+2)(\alpha+2)}{R_0^2} \int_0^{R_0/2} \frac{t^{n+3}}{\mu_{g^\ast}(B(0, t)) t^\sigma} \, \d t \\
\leq&~ C \frac{\Lambda(2\alpha+2)(\alpha+2)}{R_0^2} \left( \frac{3}{\pi} \right)^{1-n} \int_0^{R_0/2} \frac{t^{n+3}}{\omega_n t^{n+\sigma}} \, \d t \\
\leq&~ C \frac{\Lambda(2\alpha+2)(\alpha+2)}{R_0^2} \frac{(R_0/2)^{4-\sigma}}{4-\sigma}.
\end{split}
\end{align}
Therefore, combining \eqref{eq:I_1}, \eqref{eq:I_1,1}, and \eqref{eq:I_1,2}, and using $R_0 \in (\rho_0, 5R)$, we estimate $I_1$ as
\begin{equation} \label{eq:I_1_final}
I_1 \leq 2\alpha \Lambda \left( - c C_1 + C \frac{2-\sigma}{4-\sigma} (2\alpha+2)(\alpha+2) \right),
\end{equation}
for some universal constants $c, C > 0$, with $c$ independent of $\alpha$.

For $I_2$, we use a similar computation in \Cref{lem:well-definedness} to obtain
\begin{equation} \label{eq:I_2}
I_2 \leq \frac{C}{\sigma} \Lambda.
\end{equation}
On the other hand, since $v$ is bounded, by following the proof of \Cref{lem:integrability} we have
\begin{equation} \label{eq:I_3}
I_3 \leq C \Lambda \left( \frac{R}{r} \right)^\sigma \frac{2-\sigma}{1-2^{-\sigma}}
\end{equation}
for some $C = C(n, \lambda, \Lambda, \alpha) > 0$. Thus, \eqref{eq:I}, \eqref{eq:I_1_final}, \eqref{eq:I_2}, and \eqref{eq:I_3} yield
\begin{equation*}
R^\sigma \mathcal{M}^+_{\mathcal{L}_0} v (x) \leq 2\alpha \Lambda \left( - c C_1 + C \frac{2-\sigma}{4-\sigma} (2\alpha+2)(\alpha+2) \right) + \frac{C}{\sigma}\Lambda + C \Lambda \left( \frac{R}{r} \right)^\sigma \frac{2-\sigma}{1-2^{-\sigma}}.
\end{equation*}
We now choose $\sigma_0$ close to 2 so that the terms containing $(2-\sigma)$ become small. Then we obtain
\begin{equation*}
R^\sigma \mathcal{M}^+_{\mathcal{L}_0} v (x) \leq -c C_1 \Lambda + \frac{C}{\sigma_0} \Lambda,
\end{equation*}
which finishes the proof by assuming that we have taken $\alpha$ sufficiently large so that \eqref{eq:alpha} holds with $-cC_1+\frac{C}{\sigma_0}<-1$.
\end{proof}

\begin{lemma} \label{lem:barrier2}
For $z_0 \in M$, let $K$ be the supremum of the sectional curvatures in $B_{\mathrm{inj}(z_0)}(z_0)$ and let $15R < \mathrm{inj}(z_0) \land \frac{\pi}{\sqrt{K}}$. Given any $\sigma_0 \in (0,2)$, there are universal constants $\alpha > 0$ and $\kappa \in (0, 1/4]$, depending only on $n$, $\lambda$, $\Lambda$, and $\sigma_0$, such that the function
\begin{equation*}
v(x) = \max \bigg\lbrace - \left( \frac{\kappa \rho_0}{5} \right)^{-2\alpha}, - \left( \frac{d_{z_0}(x)}{5R} \right)^{-2\alpha} \bigg\rbrace
\end{equation*}
is a supersolution to
\begin{equation*}
R^{\sigma} \mathcal{M}^+_{\mathcal{L}_0} v(x) + \Lambda \leq 0,
\end{equation*}
for every $\sigma_0 <\sigma <2$ and $x \in B_{5R}(z_0) \setminus \overline{B}_{\rho_0 R}(z_0)$.
\end{lemma}

\begin{proof}
Let $\sigma_1$ and $\alpha_0$ be the $\sigma_0$ and $\alpha$ in \Cref{lem:barrier1}, respectively. If $\sigma \geq \sigma_1$, then the desired result holds with $\alpha_0$ and $\kappa = 1/4$. Now for $\sigma_0 < \sigma < \sigma_1$, we
will show the result still holds if we choose $\kappa$ small enough.

Let $\alpha = \max(\alpha_0, n/2)$. For $x$ with $R_0 = d_{z_0}(x) \in (\rho_0 R, 5R)$,
let us consider normal coordinates centered at $x$.
Then, as in \Cref{lem:barrier1} we have
\begin{align*}
R^\sigma \mathcal{M}^+_{\mathcal{L}_0} v (x)
\leq ~ & (2-\sigma) R^\sigma \int_{B(0, R_0/2)} \frac{\Lambda \delta(v, x, z)_+}{\mu_{g^\ast} (B(0, \ve \xi \ve)) \ve \xi \ve^\sigma} \, \d V_s^\ast (\xi) \\
&- (2-\sigma) R^\sigma \int_{B(0, R_0/2)} \frac{\lambda \delta(v, x, z)_-}{\mu_{g^\ast} (B(0, \ve \xi \ve)) \ve \xi \ve^\sigma} \, \d V_s^\ast (\xi) + C =: I_1 + I_2 + C.
\end{align*}
Since $v \in C^2$ in $B(x, R_0/2)$ and $v$ is bounded above, $\delta_+$ is bounded above.
Hence, $I_1 \le C$ for some universal constant. On the other hand, since $\delta_-$ is not
integrable and $\sigma < \sigma_1 < 2$, we choose $\kappa$ small enough so that $I_1 + I_2 + C < - \Lambda$.
\end{proof}

\begin{corollary} \label{cor:barrier}
For $z_0 \in M$, let $K$ be the supremum of the sectional curvatures in $B_{\mathrm{inj}(z_0)}(z_0)$ and let $15R < \mathrm{inj}(z_0) \land \frac{\pi}{\sqrt{K}}$. There is a function $v$ such that
\begin{equation*}
\begin{cases}
v \geq 0 &\text{in}~ M \setminus B_{5R}(z_0), \\
v \leq 0 &\text{in}~ B_{2R}(z_0), \\
R^\sigma \mathcal{M}^+_{\mathcal{L}_0} v + \Lambda \leq 0 &\text{in}~ B_{5R} \setminus \overline{B}_{\rho_0 R}(z_0), \\
R^\sigma \mathcal{M}^+_{\mathcal{L}_0} v \leq C &\text{in}~ B_{5R}(z_0), \\
v \geq -C &\text{in}~ B_{5R}(z_0), \\
\end{cases}
\end{equation*}
for some universal constant $C > 0$, depending only on $n$, $\lambda$, $\Lambda$, and $\sigma_0$.
\end{corollary}

\begin{proof}
Let $\alpha$ and $\kappa$ be constants given in \Cref{lem:barrier2}. We define $v(x) = \psi(d_{z_0}^2(x) / R^2)$, where $\psi$ is a smooth and increasing function on $[0, \infty)$ such that
\begin{equation*}
\psi(t) = \left( \frac{3^2}{5^2} \right)^{-\alpha} - \left( \frac{t}{5^2} \right)^{-\alpha} \quad\text{if}~ t \geq (\kappa \rho_0)^2.
\end{equation*}
By \Cref{lem:barrier2}, $R^\sigma \mathcal{M}^+_{\mathcal{L}_0} v + \Lambda \leq 0$ in $B_{5R} \setminus B_{\rho_0 R}$. Thus, it only remains to show that $R^\sigma \mathcal{M}^+_{\mathcal{L}_0} v \leq C$ in $\overline{B}_{\rho_0 R}$. Indeed, for $x \in \overline{B}_{\rho_0 R}$, we have that $\ve \delta(v, x, z) \ve \leq C d_x(z)^2/R^2$ for $z \in B_R(x)$, and that $v$ is bounded by a uniform constant. Therefore, we obtain $R^\sigma \mathcal{M}^+_{\mathcal{L}_0} v (x) \leq C $ by \Cref{lem:well-definedness}.
\end{proof}

\section{\texorpdfstring{$L^\varepsilon$}{L}-estimate} \label{sec:L_eps_estimate}

This section is devoted to the so-called $L^\varepsilon$-estimate, which is the main ingredient in the proof of the Harnack inequality. It will follow from the following lemma, which connects a pointwise estimate to an estimate in measure, and the standard Calder\'on–Zygmund technique in \cite{Cab97}.

\begin{lemma} \label{lem:base}
Let $\sigma_0 \in (0,2)$ and assume $\sigma \in [\sigma_0, 2)$. Let $c_1$, $c_2$, and $\delta_0$ be the constants in \Cref{thm:dyadic_cubes}, and let $\delta = \frac{2c_1}{c_2} \delta_0$. For $z_0 \in M$, let $K$ be the supremum of the sectional curvatures in $B_{\mathrm{inj}(z_0)}(z_0)$ and let $15R < \mathrm{inj} (z_0) \land \frac{\pi}{\sqrt{K}}$. If $u \in C^2(B_{7R}(z_0))$ is a nonnegative function on $M$ satisfying $R^\sigma \mathcal{M}^-_{\mathcal{L}_0} u \leq \varepsilon_0$ in $B_{7R}(z_0)$ and $\inf_{B_{2R}} u \leq 1$, then
\begin{equation*}
\frac{\mu_g ( \lbrace u \leq M_0 \rbrace \cap B_{\delta R}(z_0))}{\mu_g(B_{7R}(z_0))} \geq c_0,
\end{equation*}
where $\varepsilon_0 > 0$, $c_0 \in (0,1)$, and $M_0 > 1$ are constants depending only on $n$, $\lambda$, $\Lambda$, $a_1$, $a_2$, and $\sigma_0$.
\end{lemma}

\begin{proof}
Let $v$ be the barrier function constructed in \Cref{cor:barrier} and define $w = u + v$. Then $w$ satisfies $w \geq 0$ in $M \setminus B_{5R}$, $\inf_{B_{2R}} w \leq 1$, and $R^\sigma \mathcal{M}^-_{\mathcal{L}_0} w \leq \varepsilon_0 + R^\sigma \mathcal{M}^+_{\mathcal{L}_0} v$ in $B_{5R}$. By applying \Cref{lem:key} to $w$ with its envelope $\Gamma_w$, we have
\begin{equation*}
\mu_g(B_R) \leq C \sum_j \bigg( \varepsilon_0 + R^\sigma \max_{\overline{Q}_\alpha^j} \mathcal{M}^+_{\mathcal{L}_0} v + \Lambda \bigg)_+^n \mu_g(Q_\alpha^j).
\end{equation*}
Since $R^\sigma \mathcal{M}^+_{\mathcal{L}_0} v + \Lambda \leq 0$ in $B_{5R} \setminus \overline{B}_{\rho_0 R}$ and $R^\sigma \mathcal{M}^+_{\mathcal{L}_0} v \leq C$ in $B_{5R}$, we obtain
\begin{equation*}
\mu_g(B_R) \leq C \varepsilon_0^n \mu_g(B_{5R}) + C \sum_{\overline{Q}_\alpha^j \cap \overline{B}_{\rho_0 R} \neq \emptyset} \mu_g(Q_\alpha^j).
\end{equation*}
We use the volume doubling property \eqref{eq:VD} and then take $\varepsilon_0 > 0$ sufficiently small so that we have
\begin{equation*}
\mu_g(B_{7R}) \leq C \sum_{\overline{Q}_\alpha^j \cap \overline{B}_{\rho_0 R} \neq \emptyset} \mu_g(Q_\alpha^j).
\end{equation*}
By using \Cref{lem:ABP_discrete} (iv), we obtain
\begin{equation*}
\mu_g(B_{7R}) \leq C \sum_{\overline{Q}_\alpha^j \cap \overline{B}_{\rho_0 R} \neq \emptyset} \mu_g \left( B(z_\alpha^j, (1+4/\rho_1)c_2 \delta_0^j) \cap \lbrace w \leq \Gamma + C \rbrace \right).
\end{equation*}
We claim that $B(z_\alpha^j, (1+4/\rho_1)c_2 \delta_0^j) \subset B_{\delta R}(z_0)$ whenever $\overline{Q}_\alpha^j \cap \overline{B}_{\rho_0 R} \neq \emptyset$. Indeed, let $z_\ast \in \overline{Q}_\alpha^j \cap \overline{B}_{\rho_0 R}$, then for any $z \in B(z_\alpha^j, (1+4/\rho_1)c_2 \delta_0^j)$, we have
\begin{equation*}
d(z, z_0) \leq d(z, z_\alpha^j) + d(z_\alpha^j, z_\ast) + d(z_\ast, z_0) < (1+4/\rho_1)c_2 \delta_0^j + c_2 \delta_0^j + \rho_0 R.
\end{equation*}
We recall from the construction of $Q_\alpha^j$ and \eqref{eq:rho}, that $c_2 \delta_0^j \leq r_0/4 < \rho_0 R$ and
\begin{equation*}
d(z, z_0) < (3+4/\rho_1) \rho_0 R \leq \delta R,
\end{equation*}
which proves the claim. Thus, by taking a subcover of $\lbrace B(z_\alpha^j, (1+4/\rho_1)c_2 \delta_0^j) \rbrace$ with finite overlapping and using $v \geq -C$ in $B_{5R}$, we arrive at
\begin{equation*}
\mu_g(B_{7R}) \leq C \mu_g(\lbrace u \leq M_0 \rbrace \cap B_{\delta R})
\end{equation*}
for some $M_0 > 1$. Therefore, we obtain the desired result by letting $c_0 = 1/C$.
\end{proof}

Let $\delta_1 = \delta_0 (1-\delta_0)/2 \in (0,1)$. Let $k_R$ be the integer satisfying
\begin{equation*}
c_2 \delta_0^{k_R - 1} < R \leq c_2 \delta_0^{k_R - 2},
\end{equation*}
which is the generation of a dyadic cube whose size is comparable to that of some ball of radius $R$. \Cref{lem:base}, together with the Calder\'on–Zygmund technique developed in \cite{Cab97}, provides the following $L^\varepsilon$-estimate.

\begin{lemma} \label{lem:L-eps}
Let $\sigma_0 \in (0,2)$ and assume $\sigma \in [\sigma_0, 2)$. Let $\varepsilon_0$, $c_0$, and $M_0$ be the constants in \Cref{lem:base}. For $z_0 \in M$, let $R > 0$ be such that $15R < \mathrm{inj} (z_0) \land \frac{\pi}{\sqrt{K}}$. Let $u \in C^2(B_{7R}(z_0))$ be a nonnegative function on $M$ satisfying $R^\sigma \mathcal{M}^-_{\mathcal{L}_0} u \leq \varepsilon_0$ in $B_{7R}(z_0)$ and $\inf_{B(z_0, \delta_1 R)} u \leq 1$. If $Q_1$ is a dyadic cube of generation $k_R$ such that $d(z_0, Q_1) \leq \delta_1 R$, then
\begin{equation*}
\frac{\mu_g ( \lbrace u > M_0^i \rbrace \cap Q_1 )}{\mu_g(Q_1)} \leq (1 - c_0)^i.
\end{equation*}
for all $i= 1, 2, \cdots$. As a consequence, we have
\begin{equation*}
\frac{\mu_g( \lbrace u > t \rbrace \cap Q_1 )}{\mu_g(Q_1)} \leq C t^{-\varepsilon}, \quad t > 0,
\end{equation*}
for some universal constants $C > 0$ and $\varepsilon > 0$.
\end{lemma}

A simple chaining argument and \Cref{lem:L-eps} prove the following weak Harnack inequality.

\begin{theorem} [Weak Harnack inequality] \label{thm:WHI}
Let $\sigma_0 \in (0,2)$ and assume $\sigma \in [\sigma_0, 2)$. For $z_0 \in M$, let $K$ be the supremum of the sectional curvatures in $B_{\mathrm{inj}(z_0)}(z_0)$ and let $R > 0$ be such that $2R < \mathrm{inj} (z_0) \land \frac{\pi}{\sqrt{K}}$. If $u \in C^2(B_{2R}(z_0))$ is a nonnegative function satisfying $\mathcal{M}^-_{\mathcal{L}_0} u \leq C_0$ in $B_{2R}(z_0)$, then
\begin{equation*}
\left( \frac{1}{\mu_g(B_R)} \int_{B_R} u^p \, \d V(z) \right)^{1/p} \leq C \left( \inf_{B_R} u + C_0 R^\sigma \right),
\end{equation*}
where $p > 0$ and $C > 0$ are universal constants depending only on $n$, $\lambda$, $\Lambda$, $a_1$, $a_2$, and $\sigma_0$.
\end{theorem}

See, for instances, \cite[Theorem 8.1]{Cab97} for the proof of \Cref{thm:WHI}.

\section{Harnack inequality} \label{sec:Harnack}

In this section we prove the following theorem, from which \Cref{thm:Harnack} will follow. Let us recall that $\delta_1 = \delta_0 (1-\delta_0)/2 \in (0,1)$.

\begin{theorem}
Let $\sigma_0 \in (0,2)$ and assume $\sigma \in [\sigma_0, 2)$. For $z_0 \in M$, let $K$ be the supremum of the sectional curvatures in $B_{\mathrm{inj}(z_0)}(z_0)$ let $R > 0$ be such that $15R < \mathrm{inj} (z_0) \land \frac{\pi}{\sqrt{K}}$. If a nonnegative function $u \in C^2(B_{7R}(z_0))$ satisfies
\begin{equation*}
R^\sigma \mathcal{M}^-_{\mathcal{L}_0} u \leq \varepsilon_0 \quad\text{and}\quad R^\sigma \mathcal{M}^+_{\mathcal{L}_0} u \geq -\varepsilon_0 \quad\text{in}~ B_{7R}(z_0)
\end{equation*}
and $\inf_{B(z_0, \delta_1 R)} u \leq 1$, then
\begin{equation*}
\sup_{B(z_0, \delta_1 R/4)} u \leq C,
\end{equation*}
where $\varepsilon_0 > 0$ and $C > 0$ are universal constants depending only on $n$, $\lambda$, $\Lambda$, $a_1$, $a_2$, and $\sigma_0$.
\end{theorem}

\begin{proof}
Let $\varepsilon$ and $\varepsilon_0$ be the constants as in \Cref{lem:L-eps} and let $\gamma = n/\varepsilon$. Let us consider the minimal value of $t > 0$ such that 
\begin{equation} \label{eq:h_t}
u(x) \leq h_t(x) := t \left( 1 - \frac{d_{z_0}(x)}{\delta_1 R} \right)^{-\gamma} \quad \text{for all} ~ x \in B(z_0, \delta_1 R).
\end{equation}
Then by \eqref{eq:h_t}, we have $\sup_{B(z_0, \delta_1 R/4)} u \leq t (3/4)^{-\gamma}$, from which we can conclude the theorem once we show that $t \leq C$ for some universal constant.

There exists a point $x_0 \in B(z_0, \delta_1 R)$ satisfying $u(x_0) = h_t(x_0)$. Let $d = \delta_1 R - d_{z_0}(x_0)$, $r = d/2$, and $A = \lbrace u > u(x_0)/2 \rbrace$, then we have $u(x_0) = h_t(x_0) = t (\delta_1 R/d)^\gamma$. Let $Q_1$ be the unique dyadic cube of generation $k_R$ that contains the point $x_0$, which clearly satisfies $d(z_0, Q_1) \leq \delta_1 R$. Then, by applying \Cref{lem:L-eps} to $u$ with $Q_1$, we obtain
\begin{equation} \label{eq:Q_1}
\mu_g(A \cap Q_1) \leq C \left( \frac{u(x_0)}{2} \right)^{-\varepsilon} \mu_g(Q_1) \leq C t^{-\varepsilon} \left( \frac{r}{R} \right)^n \mu_g(Q_1).
\end{equation}
We will show that there is a small constant $\theta > 0$ such that
\begin{equation} \label{eq:ball_theta}
\mu_g(\lbrace u \leq u(x_0) / 2 \rbrace \cap Q_2) \leq \frac{1}{2} \mu_g(Q_2),
\end{equation}
where $Q_2 \subset Q_1$ is the dyadic cube of generation $k_{\theta r/14}$ containing the point $x_0$, provided that $t$ is sufficiently large. Recalling that $Q_2$ contains some ball $B(z, c_1 \delta_0^{k_{\theta r/14}})$, we have from \eqref{eq:Q_1} and \eqref{eq:VD} that
\begin{equation*}
\mu_g(A \cap Q_2) \leq \mu_g(A \cap Q_1) \leq \frac{C}{t^\varepsilon} \left( \frac{r}{R} \right)^n \mu_g(B(z, c_2 \delta_0^{k_R})) \leq \frac{C}{t^\varepsilon} \mu_g(B(z, c_1 \delta_0^{k_{r\theta/14}})) \leq \frac{C}{t^\varepsilon} \mu_g(Q_2).
\end{equation*}
However, when $t$ is large, we also obtain
\begin{equation*}
\mu_g(A \cap Q_2) < \frac{1}{2} \mu_g(Q_2),
\end{equation*}
which contradicts to \eqref{eq:ball_theta}. Therefore, the rest of the proof is dedicated to proving \eqref{eq:ball_theta}.

For every $x \in B(x_0, \theta r)$, we have
\begin{equation*}
u(x) \leq h_t(x) \leq t \left( \frac{d-\theta r}{\delta_1 R} \right)^{-\gamma} = \left( 1- \frac{\theta}{2} \right)^{-\gamma} u(x_0).
\end{equation*}
Let us define the functions
\begin{equation*}
v(x) := \left( 1- \frac{\theta}{2} \right)^{-\gamma} u(x_0) - u(x)
\end{equation*}
and $w := v_+$. We will apply \Cref{lem:L-eps} to $w$ in $B(x_0, 7(\theta r/14))$. For $x \in B(x_0, 7(\theta r/14))$, since $v$ is nonnegative in $B(x_0, \theta r)$, we have
\begin{align} \label{eq:Mw}
\begin{split}
\mathcal{M}^-_{\mathcal{L}_0} w(x)
&\leq \mathcal{M}_{\mathcal{L}_0}^- v(x) + \mathcal{M}_{\mathcal{L}_0}^+ v_-(x) \\
&\leq - \mathcal{M}^+_{\mathcal{L}_0} u(x) + \Lambda (2-\sigma) \int_{M \setminus B(x_0, \theta r)} \frac{v_-(z)}{\mu_g(B(x, d_x(z))) d_x(z)^\sigma} \, \d V(z) \\
&\leq R^{-\sigma} \varepsilon_0 + \Lambda (2-\sigma) \int_{M \setminus B(x_0, \theta r)} \frac{(u(z) - (1-\theta/2)^{-\gamma} u(x_0))_+}{\mu_g(B(x, d_x(z))) d_x(z)^\sigma} \, \d V(z).
\end{split}
\end{align}

Let us define a function
\begin{equation*}
g_\beta(x) := \beta \left( 1- \frac{d_{z_0}(x)^2}{R^2} \right)_+,
\end{equation*}
and consider the largest number $\beta > 0$ such that $u \geq g_\beta$. From the assumption $\inf_{B(z_0, \delta_1 R)} u \leq 1$, we have $(1-\delta_1^2) \beta \leq 1$. Let $x_1 \in B(z_0, R)$ be a point where $u(x_1) = g_\beta(x_1)$. Then, since
\begin{align*}
&(2-\sigma) \mathrm{p.v.} \int_M (u(z) - u(x_1))_- \nu_{x_1}(z) \, \d V(z) \\
&\leq (2-\sigma) \mathrm{p.v.} \int_M (g_\beta(z) - g_\beta(x_1))_- \nu_{x_1}(z) \, \d V(z) \leq C R^{-\sigma},
\end{align*}
we obtain that
\begin{align*}
\varepsilon_0
&\geq R^\sigma \mathcal{M}^-_{\mathcal{L}_0} u(x_1) \\
&= R^\sigma \inf \left( (2-\sigma) \mathrm{p.v.} \int_M \big( (u(z) - u(x_1))_+ - (u(z) - u(x_1))_- \big) \nu_{x_1}(z) \, \d V(z) \right) \\
&\geq R^\sigma \lambda (2-\sigma) \mathrm{p.v.} \int_M \frac{(u(z) - u(x_1))_+}{\mu_g(B(x_1, d_{x_1}(z))) d_{x_1}(z)^\sigma} \, \d V(z) - C.
\end{align*}
It follows from $u(x_1) \leq \beta \leq 1/(1-\delta_1^2) =: c$ that
\begin{align} \label{eq:u-c}
\begin{split}
&(2-\sigma) \mathrm{p.v.} \int_M \frac{(u(z) - c)_+}{\mu_g(B(x_1, d_{x_1}(z))) d_{x_1}(z)^\sigma} \, \d V(z) \\
&\leq (2-\sigma) \mathrm{p.v.} \int_M \frac{(u(z) - u(x_1))_+}{\mu_g(B(x_1, d_{x_1}(z))) d_{x_1}(z)^\sigma} \, \d V(z) \leq CR^{-\sigma}.
\end{split}
\end{align}

Let us now estimate the integral in \eqref{eq:Mw} by using \eqref{eq:u-c}. If $u(x_0) \leq c$, then we have $t = u(x_0) (\delta_1 R / d)^{-\gamma} \leq c \delta_1^{-\gamma}$ and we are done. Otherwise, we obtain that
\begin{align*}
&\mathcal{M}^-_{\mathcal{L}_0} w(x) \\
&\leq R^{-\sigma} \varepsilon_0 + \Lambda (2-\sigma) \int_{M \setminus B(x_0, \theta r)} \frac{(u(z) - c)_+}{\mu_g(B(x, d_x(z))) d_x(z)^\sigma} \, \d V(z) \\
&= R^{-\sigma} \varepsilon_0 + \Lambda (2-\sigma) \int_{M \setminus B(x_0, \theta r)} \frac{(u(z) - c)_+}{\mu_g(B(x_1, d_{x_1}(z))) d_{x_1}(z)^\sigma} \frac{\mu_g(B(x_1, d_{x_1}(z))) d_{x_1}(z)^\sigma}{\mu_g(B(x, d_x(z))) d_x(z)^\sigma} \, \d V(z).
\end{align*}
For $x \in B(x_0, \theta r/2)$, $x_1 \in B(z_0, R)$, and $z \in M \setminus B(x_0, \theta r)$, we have
\begin{equation*}
\frac{d_{x_1}(z)}{d_x(z)} \leq 1 + \frac{d(x, x_1)}{d(x, z)} \leq 1 + \frac{d(x, x_0) + d(x_0, z_0) + d(z_0, x_1)}{d(x_0, z) - d(x_0, x)} \leq 1 + \frac{\theta r/2 + \delta_1 R + R}{\theta r/2} \leq C \frac{R}{\theta r},
\end{equation*}
and hence, by \eqref{eq:VD},
\begin{equation*}
\frac{\mu_g(B(x_1, d_{x_1}(z))) d_{x_1}(z)^\sigma}{\mu_g(B(x, d_x(z))) d_x(z)^\sigma} \leq C \left( \frac{R}{\theta r} \right)^{n+\sigma}.
\end{equation*}
Therefore, we have shown that
\begin{equation*}
\left( \frac{\theta r}{14} \right)^\sigma \mathcal{M}^-_{\mathcal{L}_0} w \leq C \left( \frac{R}{\theta r} \right)^n
\end{equation*}
in $B(x_0, 7(\theta r/14))$.

Let $Q_2 \subset Q_1$ be the dyadic cube of generation $k_{\theta r /14}$ containing the point $x_0$. Then by \Cref{lem:L-eps}, we have
\begin{align*}
\mu_g(\lbrace u < u(x_0)/2 \rbrace \cap Q_2)
&= \mu_g(\lbrace w > \left( \left( 1-\theta/2 \right)^{-\gamma} - 1/2 \right) u(x_0) \rbrace \cap Q_2) \\
&\leq \frac{C \mu_g(Q_2)}{\left( \left( 1-\theta/2 \right)^{-\gamma} - 1/2 \right)^\varepsilon u(x_0)^\varepsilon} \left( \inf_{B(x_0, \delta_1 \theta r/14)} w + C \left( \frac{R}{\theta r} \right)^n \right)^\varepsilon.
\end{align*}
We can make the quantity $(1-\theta/2)^{-\gamma} - 1/2$ bounded away from 0 by taking $\theta > 0$ sufficiently small. Recalling that $u(x_0) = t (\delta_1 R/2r)^\gamma$, $w(x_0) = ((1-\theta/2)^{-\gamma}-1) u(x_0)$, and $\gamma = n/\varepsilon$, we obtain
\begin{equation*}
\mu_g(\lbrace u < u(x_0)/2 \rbrace \cap Q_2) \leq C \mu_g(Q_2) \left( ((1-\theta/2)^{-\gamma}-1)^\varepsilon + t^{-\varepsilon} \theta^{-n\varepsilon} \right).
\end{equation*}
We choose a constant $\theta > 0$ sufficiently small so that
\begin{equation*}
C \left((1 - \theta/2)^{-\gamma} - 1 \right)^\varepsilon \leq \frac{1}{4}.
\end{equation*}
If $t > 0$ is sufficiently large so that $C t^{-\varepsilon} \theta^{-n\varepsilon} \leq 1/4$, then we arrive at \eqref{eq:ball_theta}. Therefore, $t$ is uniformly bounded and the desired result follows.
\end{proof}

\section{H\"older estimates} \label{sec:Holder}

This section is devoted to the proof of \Cref{thm:Holder}, which will follow from \Cref{lem:Holder_0}.

\begin{lemma} \label{lem:Holder_0}
Let $\sigma_0 \in (0,2)$ and assume $\sigma \in [\sigma_0, 2)$. For $z_0 \in M$, let $K$ be the supremum of the sectional curvatures in $B_{\mathrm{inj}(z_0)}(z_0)$ let $R > 0$ be such that $15R < \mathrm{inj} (z_0) \land \frac{\pi}{\sqrt{K}}$. If $u \in C^2(B(z_0, 7R))$ is a function such that $\ve u \ve \leq \frac{1}{2}$ in $B(z_0,7R)$ and
\begin{equation*}
R^\sigma \mathcal{M}^+_{\mathcal{L}_0} u \geq - \varepsilon_0 \quad \text{and} \quad R^\sigma \mathcal{M}^-_{\mathcal{L}_0} u \leq \varepsilon_0 \quad \text{in} ~ B(z_0, 7R),
\end{equation*}
then $u \in C^\alpha$ at $z_0$ with an estimate
\begin{equation*}
\ve u(x) - u(z_0) \ve \leq CR^{-\alpha} d(x, z_0)^\alpha,
\end{equation*}
where $\alpha \in (0, 1)$ and $C > 0$ are constants depending only on $n$, $\lambda$, $\Lambda$, $a_1$, $a_2$, and $\sigma_0$.
\end{lemma}

\begin{proof}
Let $R_k := 7 \cdot 4^{-k} R$ and $B_k := B(z_0, R_k)$. It is enough to find an increasing sequence $\lbrace m_k \rbrace_{k \geq 0}$ and a decreasing sequence $\lbrace M_k \rbrace_{k \geq 0}$ satisfying $m_k \leq u \leq M_k$ in $B_k$ and $M_k - m_k = 4^{-\alpha k}$. We initially choose $m_0 = - 1/2$ and $M_0 = 1/2$ for the case $k = 0$. Let us assume that we have sequences up to $m_k$ and $M_k$. We want to show that we can continue the sequences by finding $m_{k+1}$ and $M_{k+1}$.

Let $Q_1$ be a dyadic cube of generation $k_{R_{k+1}/7}$ that contains the point $x$. In $Q_1$, either $u > (M_k + m_k)/2$ or $u \leq (M_k + m_k)/2$ in at least half of the points in measure. We suppose that
\begin{equation} \label{eq:Q_1_half}
\mu_g \left( \lbrace u > (M_k + m_k)/2 \rbrace \cap Q_1 \right) \geq \mu_g(Q_1)/2.
\end{equation}
Let us define a function
\begin{equation*}
v(x) := \frac{u(x) - m_k}{(M_k - m_k)/2},
\end{equation*}
then $v \geq 0$ in $B_k$ by the induction hypothesis. For $w := v_+$,
\eqref{eq:Q_1_half} is read as
\begin{equation} \label{eq:Q_1_half_w}
\mu_g( \lbrace w > 1 \rbrace \cap Q_1 ) \geq \mu_g(Q_1) /2.
\end{equation}
To apply \Cref{lem:L-eps} to $w$, we need to estimate $\mathcal{M}^-_{\mathcal{L}_0} w \leq \mathcal{M}^-_{\mathcal{L}_0} v + \mathcal{M}^+_{\mathcal{L}_0} v_-$. We know that $R^\sigma \mathcal{M}^-_{\mathcal{L}_0} v \leq 2\varepsilon_0/(M_k - m_k)$ in $B_{7R}$. For $\mathcal{M}^+_{\mathcal{L}_0} v^-$, we use the bound 
\begin{equation*}
v(z) \geq -2((d_{z_0}(z) / R_k)^\alpha - 1) \quad \text{for} ~ z \in M \setminus B_k,
\end{equation*}
which follows from the definition of $v$ and the properties of sequences $M_k$ and $m_k$. Then we have, for $x \in B(z_0, 3R_{k+1})$,
\begin{align*}
\mathcal{M}^-_{\mathcal{L}_0} w(x) 
&\leq \frac{2\varepsilon_0}{M_k - m_k} R^{-\sigma} + \Lambda(2-\sigma) \int_{M \setminus B_k} \frac{v_-(z)}{\mu_g(B(x, d_x(z))) d_x(z)^\sigma} \, \d V(z) \\
&\leq \frac{2\varepsilon_0}{M_k - m_k} R^{-\sigma} + 2\Lambda (2-\sigma) \int_{M \setminus B(x, R_{k+1})} \frac{(d_{z_0}(z) / R_k)^\alpha - 1}{\mu_g(B(x, d_x(z))) d_x(z)^\sigma} \, \d V(z).
\end{align*}
Note that $d_{z_0}(z) \leq d_x(z) + d_x(z_0) \leq d_x(z) + 3R_{k+1} \leq 4d_x(z)$. Thus, by using \eqref{eq:VD} and assuming $\alpha < \sigma_0$, we obtain that
\begin{align*}
&\int_{M \setminus B(x, R_{k+1})} \frac{(d_{z_0}(z) / R_k)^\alpha - 1}{\mu_g(B(x, d_x(z))) d_x(z)^\sigma} \, \d V(z) \\
&\leq \int_{M \setminus B(x, R_{k+1})} \frac{(d_x(z) / R_{k+1})^\alpha - 1}{\mu_g(B(x, d_x(z))) d_x(z)^\sigma} \, \d V(z) \\
&\leq \sum_{j=0}^\infty \int_{B(x, 2^{j+1} R_{k+1}) \setminus B(x, 2^j R_{k+1})} \frac{2^{(j+1)\alpha} - 1}{\mu_g(B(x, 2^j R_{k+1})) (2^j R_{k+1})^\sigma} \, \d V(z) \\
&\leq \sum_{j=0}^\infty \frac{2^{(j+1)\alpha} - 1}{(2^j R_{k+1})^\sigma} \frac{\mu_g(B(x, 2^{j+1} R_{k+1}))}{\mu_g(B(x, 2^j R_{k+1}))} \\
&\leq \frac{2^n}{R_{k+1}^\sigma} \sum_{j=0}^\infty \left( 2^\alpha 2^{j(\alpha-\sigma)} - 2^{-j\alpha} \right) \\
&= \frac{2^n}{R_{k+1}^\sigma} \frac{2^\alpha - 1}{(1-2^{\alpha-\sigma})(1-2^{-\sigma})} \\
&\leq \frac{2^n}{R_{k+1}^\sigma} \frac{2^\alpha - 1}{(1-2^{\alpha-\sigma_0})(1-2^{-\sigma_0})} =: \frac{c_1(n, \alpha, \sigma_0)}{R_{k+1}^\sigma}.
\end{align*}
The constant $c_1(n, \alpha, \sigma_0)$ can be made arbitrarily small by taking small $\alpha$. We have estimated
\begin{equation*}
\left( \frac{R_{k+1}}{7} \right)^\sigma \mathcal{M}^-_{\mathcal{L}_0} w \leq C \left( \varepsilon_0 + c_1 \right)
\end{equation*}
in $B(x, 7(R_{k+1}/7))$ for $x \in B(z_0, 2R_{k+1})$. Therefore, by \Cref{lem:L-eps} and \eqref{eq:Q_1_half_w}, we have
\begin{align*}
\frac{1}{2} \mu_g(Q_1) \leq \mu_g( \lbrace w > 1 \rbrace \cap Q_1) \leq C \mu_g(Q_1) \left( w(x) + C(\varepsilon_0 + c_1) \right)^\varepsilon,
\end{align*}
or equivalently, 
\begin{equation*}
\theta \leq w(x) + C(\varepsilon_0 + c_1)
\end{equation*}
for some universal constant $\theta > 0$. We take $\varepsilon_0$ and $\alpha$ sufficiently small so that $C\varepsilon_0 < \theta/4$ and $Cc_1 < \theta /4$, then we have $w \geq \theta /2$ in $B(z_0, 2R_{k+1})$. Thus, if we set $M_{k+1} = M_k$ and $m_{k+1} = M_k - 4^{-\alpha(k+1)}$, then 
\begin{equation*}
M_{k+1} \geq u \geq m_k + \frac{M_k - m_k}{4} \theta = M_k - \left(1 - \frac{\theta}{4} \right) 4^{-\alpha k} \geq m_{k+1}
\end{equation*}
in $B_{k+1}$.

On the other hand, if $\mu_g( \lbrace u \leq (M_k + m_k)/2 \rbrace \cap Q_1) \geq \mu_g(Q_1)/2$, we define
\begin{equation*}
v(x) = \frac{M_k - u(x)}{(M_k - m_k)/2}
\end{equation*}
and continue in the same way using that $R^\sigma \mathcal{M}^+_{\mathcal{L}_0} u \geq - \varepsilon_0$.
\end{proof}

\section*{Acknowledgement}

The research of Jongmyeong Kim is supported by the National Research Foundation of Korea(NRF) grant funded by the Korea government(MSIT)(2016R1E1A1A01941893). The research of Ki-Ahm Lee is supported by the National Research Foundation of Korea(NRF) grant funded by the Korea government(MSIP): NRF-2020R1A2C1A01006256.


\begin{thebibliography}{10}

\bibitem{AOCM18}
D.~Alonso-Or\'{a}n, A.~C\'{o}rdoba, and A.~D. Mart\'{\i}nez.
\newblock Integral representation for fractional {L}aplace--{B}eltrami operators.
\newblock {\em Adv. Math.}, 328:436--445, 2018.

\bibitem{AE00}
D.~Applebaum and A.~Estrade.
\newblock Isotropic {L}\'{e}vy processes on {R}iemannian manifolds.
\newblock {\em Ann. Probab.}, 28(1):166--184, 2000.

\bibitem{BGS15}
V.~Banica, M.~d.~M. Gonz\'{a}lez, and M.~S\'{a}ez.
\newblock Some constructions for the fractional {L}aplacian on noncompact
  manifolds.
\newblock {\em Rev. Mat. Iberoam.}, 31(2):681--712, 2015.

\bibitem{BM20}
A.~Biswas and M.~Modasiya.
\newblock Regularity results of nonlinear perturbed stable-like operators.
\newblock {\em arXiv preprint arXiv:2004.06996}, 2020.

\bibitem{Cab97}
X.~Cabr\'{e}.
\newblock Nondivergent elliptic equations on manifolds with nonnegative curvature.
\newblock {\em Comm. Pure Appl. Math.}, 50(7):623--665, 1997.

\bibitem{CS07}
L.~Caffarelli and L.~Silvestre.
\newblock An extension problem related to the fractional {L}aplacian.
\newblock {\em Comm. Partial Differential Equations}, 32(7-9):1245--1260, 2007.

\bibitem{CS09}
L.~Caffarelli and L.~Silvestre.
\newblock Regularity theory for fully nonlinear integro-differential equations.
\newblock {\em Comm. Pure Appl. Math.}, 62(5):597--638, 2009.

\bibitem{CKW19}
Z.-Q. Chen, T.~Kumagai, and J.~Wang.
\newblock Elliptic {H}arnack inequalities for symmetric non-local {D}irichlet forms.
\newblock {\em J. Math. Pures Appl. (9)}, 125:1--42, 2019.

\bibitem{CY75}
S.~Y. Cheng and S.~T. Yau.
\newblock Differential equations on {R}iemannian manifolds and their geometric applications.
\newblock {\em Comm. Pure Appl. Math.}, 28(3):333--354, 1975.

\bibitem{CLN06}
B.~Chow, P.~Lu, and L.~Ni.
\newblock {\em Hamilton's {R}icci flow}, volume~77 of {\em Graduate Studies in Mathematics}.
\newblock American Mathematical Society, Providence, RI; Science Press Beijing, New York, 2006.

\bibitem{Chr90}
M.~Christ.
\newblock A {$T(b)$} theorem with remarks on analytic capacity and the {C}auchy integral.
\newblock {\em Colloq. Math.}, 60/61(2):601--628, 1990.

\bibitem{Gig12}
N.~Gigli.
\newblock Second order analysis on {$(\mathscr{P}_2(M),W_2)$}.
\newblock {\em Mem. Amer. Math. Soc.}, 216(1018):xii+154, 2012.

\bibitem{GS12}
N.~Guillen and R.~W. Schwab.
\newblock Aleksandrov--{B}akelman--{P}ucci type estimates for integro-differential equations.
\newblock {\em Arch. Ration. Mech. Anal.}, 206(1):111--157, 2012.

\bibitem{Jos17}
J.~Jost.
\newblock {\em Riemannian geometry and geometric analysis}.
\newblock Universitext. Springer, Cham, seventh edition, 2017.

\bibitem{KL20}
M.~Kim and K.-A. Lee.
\newblock Regularity for fully nonlinear integro-differential operators with kernels of variable orders.
\newblock {\em Nonlinear Anal.}, 193:111312, 2020.

\bibitem{Kim04}
S.~Kim.
\newblock Harnack inequality for nondivergent elliptic operators on {R}iemannian manifolds.
\newblock {\em Pacific J. Math.}, 213(2):281--293, 2004.

\bibitem{KKL14}
S.~Kim, S.~Kim, and K.-A. Lee.
\newblock Harnack inequality for nondivergent parabolic operators on {R}iemannian manifolds.
\newblock {\em Calc. Var. Partial Differential Equations}, 49(1-2):669--706, 2014.

\bibitem{KKL16}
S.~Kim, Y.-C. Kim, and K.-A. Lee.
\newblock Regularity for fully nonlinear integro-differential operators with regularly varying kernels.
\newblock {\em Potential Anal.}, 44(4):673--705, 2016.

\bibitem{KL14}
S.~Kim and K.-A. Lee.
\newblock Parabolic {H}arnack inequality of viscosity solutions on {R}iemannian manifolds.
\newblock {\em J. Funct. Anal.}, 267(7):2152--2198, 2014.

\bibitem{Pet16}
P.~Petersen.
\newblock {\em Riemannian geometry}, volume 171 of {\em Graduate Texts in Mathematics}.
\newblock Springer, Cham, third edition, 2016.

\bibitem{SC92}
L.~Saloff-Coste.
\newblock Uniformly elliptic operators on {R}iemannian manifolds.
\newblock {\em J. Differential Geom.}, 36(2):417--450, 1992.

\bibitem{Vil09}
C.~Villani.
\newblock {\em Optimal transport}, volume 338 of {\em Grundlehren der Mathematischen Wissenschaften [Fundamental Principles of Mathematical Sciences]}.
\newblock Springer-Verlag, Berlin, 2009.
\newblock Old and new.

\bibitem{WZ13}
Y.~Wang and X.~Zhang.
\newblock An {A}lexandroff--{B}akelman--{P}ucci estimate on {R}iemannian manifolds.
\newblock {\em Adv. Math.}, 232:499--512, 2013.

\bibitem{Xu18}
S.~Xu.
\newblock Local estimate on convexity radius and decay of injectivity radius in a {R}iemannian manifold.
\newblock {\em Commun. Contemp. Math.}, 20(6):1750060, 19, 2018.

\bibitem{Yau75}
S.~T. Yau.
\newblock Harmonic functions on complete {R}iemannian manifolds.
\newblock {\em Comm. Pure Appl. Math.}, 28:201--228, 1975.

\end{thebibliography}
\end{document}